\documentclass[a4paper]{article}

\usepackage{booktabs}
\usepackage{hyperref}
\usepackage{enumitem}

\usepackage{xltabular, multirow}
\usepackage[usestackEOL]{stackengine}

\usepackage{authblk}

\usepackage{amssymb,amsmath,amsthm,amsfonts,amscd,amstext,latexsym,makeidx,mathtools}

\newcommand{\F}{\mathbb{F}}
\newcommand{\SF}{\mathbb{S}}
\newcommand{\PSF}{\mathbb{PS}}
\newcommand{\Tr}{\mathrm{Tr}}

\newcommand{\adeg}{\mathrm{d}^{\circ}}

\newtheorem{theorem}{Theorem}
\newtheorem{lemma}{Lemma}

\newtheorem{remark}{Remark}

\newcommand{\numberlist}[2][\linewidth]{%
  \parbox[t]{#1}{\printcommalist{#2}}%
}
\newcommand{\printcommalist}[1]{%
  \begingroup\lccode`~=`,\lowercase{\endgroup\def~}{\mathcomma\penalty0 }%
  \mathcode`,="8000
  \thinmuskip=6mu plus 6mu minus 2mu
  \tiny
  $#1$
}
\mathchardef\mathcomma=\mathcode`,

\usepackage[
    sorting=nyt,
    maxbibnames=99,
    giveninits=true,
    doi = true,
    url = true
]{biblatex}

\bibliography{Bib}

\renewbibmacro*{doi+eprint+url}{%
  \iftoggle{bbx:doi}
    {\iffieldundef{url}{\printfield{doi}}{}}
    {}%
  \newunit\newblock
  \iftoggle{bbx:eprint}
    {\usebibmacro{eprint}}
    {}%
  \newunit\newblock
  \iftoggle{bbx:url}
    {\usebibmacro{url+urldate}}
    {}%
}

\DeclareSourcemap{
  \maps[datatype=bibtex, overwrite]{
    \map{
      \step[fieldset=editor, null]
    }%
  }%
}

\title{On a classification of planar functions in characteristic three}
\author[1]{Samuele Andreoli}
\author[1]{Lilya Budaghyan}
\author[2]{Robert Coulter}
\author[1]{Alise Haukenes}
\author[1]{Nikolay Kaleyski}
\author[1]{Enrico Piccione}
\affil[1]{University of Bergen, Norway, \{name.surname\}@uib.no}
\affil[2] {Department of Mathematical Sciences, University of Delaware, Newark, DE, 19716, USA, coulter@udel.edu}

\date{}

\begin{document}

\maketitle
\begin{abstract}


Planar functions are functions over a finite field that have optimal combinatorial properties and they have applications in several branches of mathematics, including algebra, projective geometry and cryptography. There are two relevant equivalence relations for planar functions, that are isotopic equivalence and CCZ-equivalence. Classification of planar functions is performed via CCZ-equivalence which arises from cryptographic applications. In the case of quadratic planar functions, isotopic equivalence, coming from connections to commutative semifields, is more general than CCZ-equivalence and isotopic transformations can be considered as a construction method providing up to two CCZ-inequivalent mappings. In this paper, we first survey known infinite classes and sporadic cases of planar functions up to CCZ-equivalence, aiming to exclude equivalent cases and to identify those with the potential to provide additional functions via isotopic equivalence. In particular, for fields of order $3^n$ with $n\le 11$, we completely resolve if and when isotopic equivalence provides different CCZ-classes for all currently known planar functions. Further, we perform an extensive computational investigation on some of these fields and find seven new sporadic planar functions over $\F_{3^6}$ and two over $\F_{3^9}$. Finally, we give new simple quadrinomial representatives for the Dickson family of planar functions.

\end{abstract}


\section{Introduction}
Functions over finite fields have numerous applications in many areas of research within mathematics and computer science.
Notably, they have been heavily studied for the design of cryptographic ciphers.
For many years, the design of cryptographic ciphers has been mostly based on binary fields, due to the efficiency of implementing binary operations in hardware and software. More recently there has been a renewed interest in the case of odd characteristic, especially for building MPC-friendly cryptography primitives and in the field of side-channel countermeasures, see Grassi, Rechberger, Rotaru, Scholl and Smart \cite{GRRSS-2016-mfskp} or Masure, M\' eaux, Moos and Standaert \cite{MMMS-2023-eaemi} for example.
Thus, it seems desirable to investigate functions over finite fields of odd characteristic with useful cryptographic properties.
In 1989, Meier and Staffelbach \cite{MS-1989-nlccs} showed the importance of using highly nonlinear functions in designing cryptographic primitives.
Moreover, after the discovery of differential cryptanalysis by Biham and Shamir \cite{BS-1991-dcodl} in 1991,
Nyberg \cite{N-1993-dumic} introduced the mathematical concept of differential uniformity of a function between
finite fields with the idea that functions with optimal differential uniformity would provide the strongest resistance to differential cryptanalysis.
In binary fields, the optimal differential uniformity is $2$, and functions obtaining this optimal value are called Almost Perfect Nonlinear (APN) functions. Over fields of odd characteristics, the optimal differential uniformity is 1, and functions that obtain this value are called Perfect Nonlinear (PN) or planar functions.
The applications of planar functions, however, are not limited to their possible uses in cryptography. Indeed, when Dembowski and Ostrom \cite{DO-1968-poowc} first introduced planar functions in 1968 as tools for constructing projective planes admitting collineation groups with specific properties. Additionally, Coulter and Henderson \cite{CH-2008-cpas} established a 1-to-1 correspondence between commutative presemifields of odd order and quadratic planar functions. This connection was subsequently used by Budaghyan and Helleseth \cite{BH-2011-ncsdb} to construct the first new infinite family of commutative semifields of arbitrary odd characteristic since the 1950s.


Information about the known families and constructions of planar functions (and commutative semifields) is scattered. There has not been a serious attempt to provide an exhaustive survey of the known classes over small fields since 2010, when Pott and Zhou \cite{PZ-2010-scopf} gave an
account of what was known up to CCZ-equivalence for fields of order $3^n$ with $n\le 6$. There is also some confusion over the historical record, in part because some papers introducing ``new" examples of planar functions have failed to fully address the question of equivalence with known classes.
A central motivation of this paper is to rectify the above two problems by presenting, as best we can, a complete picture of the current state of knowledge with regards to planar functions over fields of order $3^n$ with $n\le 11$. As part of this account, we give a nomenclature of the known families of planar functions, which is as accurate as it can be. With this, we hope to provide a complete account of the current situation, including whatever is known about the intersections between the known families. Following a section on background material and notation, this nomenclature is given in Section 3 along with explanations and remarks.

In Section 4 we introduce new planar DO polynomials representing the Dickson semifields. The best previous representatives for the entire class were all hexanomials. We use a linear transformation on that previous class to reduce them to quadrinomials. Our motivation stems from the fact that most computations involving equivalence are made easier by two things, namely reducing the number of terms of the polynomials involved and restricting the coefficients of the polynomials to as small a subfield as possible.

Finally, in Section 5 we turn to our central goal of giving a complete account of the current knowledge of planar functions, up to CCZ-equivalence, over fields of order $3^n$ with $n\le 11$. 
The latest effort to produce such a classification, to the best of our knowledge, was made by Pott and Zhou in 2010 \cite{PZ-2010-scopf}, who gave a classification up to CCZ-equivalence of all planar functions known at the time in characteristic 3, for dimensions up to 6. We update this classification, adding the representatives of families and sporadic instances discovered in the intervening years.
Testing for equivalence in dimensions larger than 6 has been exceedingly difficult until the recent discovery of a more efficient equivalence test by Ivkovic and Kaleyski \cite{IK-2022-darle}.
Using this improved equivalence test, we have expanded the classification up to dimension 11. Notably, we find that the Coulter-Henderson-Cosick(CHK) semifield \cite{CHK-2007-ppfcs}, considered a sporadic instance until now, belongs to the Budaghyan-Helleseth family \cite{BH-2008-npnmo}.
We classify planar functions up to CCZ-equivalence because it is the natural equivalence relation to consider when investigating differential properties.
However, for quadratic planar functions, there is a more general equivalence relation based on the isotopy of the corresponding semifields, namely isotopic equivalence.
The families presented in Section 3 are equivalent with respect to isotopic equivalence. This means that a complete classification, consistent with the list of families provided is only possible if we investigate the number of CCZ-classes (or, equivalently, strong isotopic equivalence classes), inside the isotopic equivalence classes.
The Coulter-Henderson Theorem \cite{CH-2008-cpas} provides necessary conditions for an isotopic equivalence class to split into two strong isotopic equivalence classes, as well as conditions on the form of isotopisms that are not also strong isotopisms.
Using the latter, we can computationally prove whether an isotopy class splits into two strong isotopy classes for $n\le 11$.
We do so either by finding a split or by exhausting the search space for possible isotopisms. We find that some of the instances of the Budaghyan-Helleseth (BH) family and the instances of the Cohen-Ganley (CG) family over $\F_{3^8}$ define isotopic equivalence classes that split into two strong isotopic equivalence classes. To the best of our knowledge, this was not previously known. The complete account, along with tables of the invariants used, is reported in the tables of Appendix \ref{app:main}.
With this in place, we proceed to search for new planar functions up to CCZ-equivalence.
Our approach is to perform an expansion search, a procedure analogous to the one employed by Aleksandersen, Budaghyan and Kaleyski \cite{ABK-2022-sfafb} for APN functions. This procedure, and the complete results of the searches, are outlined in Section 5.1.
As a result of this search, we find seven new CCZ-inequivalent examples of sporadic planar functions over $\F_{3^6}$ and two over $\F_{3^9}$.
Finding so many new examples of planar functions over such a low dimension is a little surprsing, and we believe a more involved computational search might reveal even more instances. Additionally, we find that some of the new instances belong to an isotopy class that splits into two CCZ-equivalence classes.

\section{Preliminaries}

This section introduces the necessary background on planar function and semifields, as well as establishing notation. We refer the reader to Budaghyan \cite[Chapter 2]{B-2014-caaoc} for further details.

\subsection{Planar functions}

Let $p$ be an odd prime and $n$ be a positive integer. We denote by $\F_{p^n}$ the finite field with ${p^n}$ elements, and by $\F_p^n$ the $n$-dimensional vector space over $\F_p$. It is well known that one can identify the vector space $\F_{p}^n$ with the field $\F_{p^n}$ through the use of a basis over $\F_{p}$. Let $k$ be a divisor of $n,$ we denote the trace function of $\F_{p^n}$ over $\F_{p^k}$ as $\Tr^n_k(x)=x+x^{p^k}+\cdots +x^{p^{n-k}}$. Moreover, we denote by $\Tr=\Tr_n=\Tr^{n}_1$ the absolute trace function. 

Any function $F\colon\F_p^n\to\F_p^n$ can be represented uniquely as a univariate polynomial $F(x) \in \F_{p^n}\left[X\right]$ of the form
\begin{equation}\label{eq:univariate}
    F(x)=\sum_{i=0}^{p^n-1}\alpha_i x^i,
\end{equation}
where $\alpha_i \in \F_{p^n}$. For any positive integer $i$ strictly less than $p^n$, we set $\mathrm{d}_p(i)$ to be equal to $\sum_{j=0}^{n-1}c_j$ for an appropriate choice of $0 \leq c_j < p$ such that $i = \sum_{j=0}^{n-1}c_jp^j$.
We denote by $\adeg(F)$ the \emph{algebraic degree} of $F$, that is defined as
$$\adeg(F)=\max\{\mathrm{d}_p(i)\colon 0\leq i<p^n,\,\alpha_i\neq0 \}$$
where the $\alpha_i$ are as in \eqref{eq:univariate}.
If $F$ has algebraic degree at most $1$, then $F$ is called \emph{affine} and if additionally $F(0)=0$ we also say that $F$ is \emph{linear}. If $F$ has algebraic degree $2$ (resp. $3$), then $F$ is said to be \emph{quadratic} (resp. \emph{cubic}). The function $F$ is called \emph{Dembrowski-Ostrom (DO) polynomial} if all the non-zero terms in its polynomial representation have algebraic degree $2$, that is
$$F(x) = \sum_{0\leq i\leq j<p^n} \alpha_{i,j} x^{p^j+p^i}$$
where $\alpha_{i,j}\in\F_{p^n}$.
We denote by $D_aF(x)$ the \emph{derivative} of $F$ in direction $a\in \F_{p^n}\setminus\{0\}$,
$$D_aF(x)=F(x+a)-F(x).$$
Let $\delta$ be a positive integer. A function $F$ is differentially $\delta$-uniform if 
$$\delta \geq \max_{a,b\in \F_{p^n},\, a\neq 0}|\{x \in \F_{p^n} \colon D_aF(x)=b\}|.$$
A $1$-uniform function is also called a \emph{Perfect Nonlinear (PN)} function or \emph{planar} function. In the case of $F$ being a DO polynomial, then $F$ is planar if and only if for all $a\in\F_{p^n}\setminus\{0\}$ the only solution to the equation $D_aF(x)=D_aF(0)$ is $x=0.$

\subsection{Semifields}

Let $S$ be a nonempty set and $+,\star$ be two binary operations over $S$. Then $(S,+,\star)$ is called a \emph{presemifield} if the following holds
\begin{itemize}
    \item $(S,+)$ is an Abelian group; 
    \item For any $a,b\in S$, we have that  $(a+b) \star c = (a \star c) + (b \star c)$ and that $a \star (b+c) = (a \star b) + (a \star c)$;
    \item If $a,b,c\in S$ are such that $a\star b=0$, then $a=0$ or $b=0$.
\end{itemize}
If there exists $1_S\in S$ such that $1_S \star a = a \star 1_S = a$ for all $a\in S$, then $(S,+,\star)$ is called a \emph{semifield}. In 1965, Knuth \cite{K-1965-fsapp} showed that the additive group of a presemifield is isomorphic to the additive group of a finite field $\F_q$. So we can always represent a presemifield as $(\F_q,+,\star)$, where $+$ is the usual addition over the finite field $\F_q$. Moreover, any finite field is also a semifield with the usual multiplication.

Two presemifields $\PSF_1 = (\F_q,+,\star)$ and $\PSF_2 = (\F_q,+,*)$ are called \emph{isotopic} if there are three linear permutations $L$, $M$, and $N$ of $\F_q$ such that $ L(x\star y)=M(x) * N(y)$ for any $x,y\in \F_q$. If $M=N$, then $\PSF_1$ and $\PSF_2$ are called \emph{strongly isotopic}. Every finite commutative presemifield $\PSF=(\F_q,+,\star)$ is isotopic to a finite commutative semifield $\SF=(\F_q,+,*)$ where we choose any $a\in \F_q\setminus\{0\}$ and we set
$$(x\star a)*(y\star a)=x\star y.$$
Then $1_{\SF}=a\star a$. We observe that $\PSF$ and $\SF$ are strongly isotopic by using the transformation $(L,M,N)=(\mathrm{id},\mathrm{id}\star a,\mathrm{id}\star a)$ where $\mathrm{id}$ is the identity function over $\F_q$.

The \emph{left, middle} and \emph{right nucleus} of a finite semifield $\SF = (\F_q,+,\star)$ are denoted by, respectively,
\begin{align*}
    N_\ell(\SF) & = \left\{\alpha\in\F_q \colon (\alpha\star x) \star y = \alpha \star (x \star y) \text{ for all } x,y \in \F_q\right\},\\
    N_m(\SF) & = \left\{\alpha\in\F_q \colon (x\star \alpha) \star y = x \star (\alpha \star y) \text{ for all } x,y \in \F_q\right\},\\
    N_r(\SF) & = \left\{\alpha\in\F_q \colon (x\star y) \star \alpha = x \star (y \star \alpha) \text{ for all } x,y \in \F_q\right\}.
\end{align*}
We denote by $N(\SF)=N_\ell(\SF )\cap N_m(\SF) \cap N_r(\SF)$ the \emph{nucleus} of $\SF$. If $\SF_1$ and $\SF_2$ are isotopic, then their nuclei (left, middle, right, and the nucleus) have the same order. We have that if $\SF$ is commutative, then $N(\SF)=N_\ell(\SF )= N_r(\SF)$ because $N_\ell(\SF )= N_r(\SF)$ and $N_\ell(\SF )\subseteq N_m(\SF)$.

When $q$ is odd, Coulter and Henderson \cite{CH-2008-cpas} showed there is a 1-to-1 correspondence between finite commutative semifields and planar DO polynomials which can be realised in the following way. From any finite commutative presemifield $(\F_q,+\star)$ we can obtain a planar DO polynomial $F\in\F_q\left[X\right]$ by $F(x)=\frac{1}{2} (x\star x)$. Conversely, any planar DO polynomial $F\in\F_q\left[X\right]$ defines a commutative presemifield $(\F_q,+,\star)$ via field addition and multiplication given by $x\star_F y = F(x+y)-F(x)-F(y)$.

\subsection{Equivalence relations of functions}\label{sec:preliminaries:equivalence}

Let $F$ and $G$ be two functions over $\F_{p^n}$. The functions $F$ and $G$ are said to be:
\begin{itemize}
    \item \emph{affine equivalent} (resp. \emph{linear equivalent}) if there are two affine (resp. linear) permutations $A_1$ and $A_2$ over $\F_{p^n}$ such that $G = A_1 \circ F \circ A_2$;
    \item \emph{Extended Affine equivalent}, or \emph{EA-equivalent,} if there is an affine function $A$ over $\F_{p^n}$ such that $G$ is affine equivalent to $F+A$.
    \item \emph{Carlet, Charpin and Zinoviev equivalent}, or \emph{CCZ-equivalent,} if there is an affine permutation $\mathcal{A}$ over $\F_{p^n}^2$ mapping $\Gamma_F$ to $\Gamma_G$. Where by $\Gamma_F$ we denote the graph of the function $F$, $\Gamma_F=\left\{(x, F(x))\colon x \in \F_{p^n}\right\}$.
\end{itemize}
We specify that CCZ-equivalence is strictly more general than EA-equivalence (any pair of functions that are EA-equivalent are also CCZ-equivalent, but not necessarily vice-versa); EA-equivalence is strictly more general than affine equivalence; and affine equivalence is strictly more general than linear equivalence.

Dempwolff \cite{U-2018-ceopf} proved that two monomial functions $F(x)=x^d$ and $G(x)=x^e$ are CCZ-equivalent if and only if they are cyclotomic equivalent, that is there exists a positive integer $i$ such that $d=p^ie\pmod{p^n-1}$ or $d=p^ie^{-1}\pmod{p^n-1}$ if $e$ is invertible modulo $p^n-1$. Budaghyan and Helleseth \cite{BH-2008-npnmo} proved that the equivalence relations linear, affine, EA, and CCZ all coincide for planar DO polynomials over $\F_{p^n}$. Two planar DO polynomials $F$ and $G$ are strongly isotopic if and only if they are CCZ-equivalent, and any isotopic class contains at most two CCZ-equivalence classes. The following theorem gives more insights on this topic.
\begin{theorem}[{Coulter Henderson \cite{CH-2008-cpas}}]\label{thm:CH}
    Let $\PSF$ and $\PSF'$ be two commutative presemifields where $p$ is a prime and $n$ is a positive integer. Let $\SF=(\F_{p^n},+,\star)$ and $\SF'=(\F_{p^n},+,*)$ be two commutative semifields strongly isotopic respectively to $\PSF$ and $\PSF'$. Suppose that $\PSF$ and $\PSF'$ are isotopic. Let $|N_m(\SF)|=p^{m}$ and $|N(\SF)|=p^{k}$. Then one of the following holds
    \begin{itemize}
        \item If $m/k$ is odd, then $\PSF$ and $\PSF'$ are strongly isotopic.
        \item If $m/k$ is even, then either $\PSF$ and $\PSF'$ are strongly isotopic or all isotopisms $(L,M,N)$ from $\SF$ to $\SF'$ are such that $M=\alpha \star N$ where $\alpha\in N_m(\SF)$ is a non-square (w.r.t. to the multiplication $\star$) or, equivalently, $\alpha\in N_m(\SF)\setminus N(\SF)$. 
    \end{itemize}
    Moreover, if $n$ is odd, the notion of strong isotopism coincides with the one of isotopism for commutative presemifields.
\end{theorem}
One important implication of this theorem is that if the dimension of the middle nucleus over the left nucleus of a commutative semifield is even, then the isotopy class of the semifield can split into two strong isotopy classes. Consequently, we have checked all situations where this might be possible as part of our accounting of the known classes of planar polynomials that we give below. While this may sound computationally difficult, in practice it is not, as the specific format of the splitting scenario outlined in the theorem means we need only check the isotopisms $(L,M,N)=(x,\alpha\star x,x)$ with $\alpha\in N_m(\SF)\setminus N(\SF)$ a non-square.
This observation is not new, having already been used by Zhou
\cite{Z-2012-anoti} for example. Moreover, one can improve the computational investigation by taking only one representative $\alpha$ in each coset $\beta\star N(\SF)$ where $\beta\in N_m(\SF)\setminus N(\SF)$ non-square. Indeed, for any $\gamma\in N(\SF)\setminus\{0\}$ the two isotopism $(L,M,N)=(x,\beta\star x,x)$ and $(L,M,N)=(x,(\gamma\star \beta)\star x,x)$ give strongly isotopic semifields because
\begin{align*}
    \left((\gamma\star \beta)\star x\right)\star y=\left(\gamma\star (\beta\star x)\right)\star y=\gamma\star\left( (\beta\star x)\star y\right). 
\end{align*}

A property that is preserved by an equivalence relation is called an \emph{invariant}. In particular, if a property is preserved by linear equivalence (respectively, affine, EA, CCZ, isotopic), then it is a linear invariant (respectively, affine, EA, CCZ, isotopic). Invariants can be useful to speed up equivalence tests because functions with a different value for an invariant must be inequivalent. Not many invariants are known for planar DO polynomials $F$ over $\F_q$, but some have been useful in past research. Firstly, the orders of the various nuclei of any commutative semifield $\SF$ constructed from $F$ are isotopic invariants. Now define a linear code $\mathcal{C}_F$ with generator matrix
$$
M_F = \left[\begin{matrix}
1\\
x \\
F(x)
\end{matrix}\right]_{x\in \F_{q}}.
$$
The \emph{monomial automorphism group}, or automorphism group for short, is the group of monomial matrices that map the code $\mathcal{C}_F$ to itself. A monomial matrix is an invertible matrix where each column has exactly one non-zero entry. Pott and Zhou \cite{PZ-2010-scopf} showed that the order of the monomial automorphism group of $\mathcal{C}_F$ is a CCZ-invariant.
Finally, we consider the set of \emph{linear self-equivalences} of $F$, $EQ(F,F)$, that is the set of pairs of linear permutations $(L_1, L_2)$ such that $L_1 \circ F \circ L_2 = F$. The \emph{right orbit} of $x\in\F_q$ with respect to the set of linear self-equivalences is defined as the set
$$RO_{F,x} = \left\{y \in \F_q\colon (L_1,L_2)\in EQ(F,F), L_2(x)=y\right\}.$$
The multiset of the cardinalities of each orbit is a linear invariant \cite{IK-2022-darle}.

\section{On the known families of Planar Functions}\label{sec:families}

In this section we address the nomenclature of planar functions.
Our aim is to clarify any confusion there may be in the literature regarding
precedent and discovery while at the same time explaining our naming convention
and hopefully standardizing the names of known families.
We do this in part to address the recent proliferation of claims of new planar
functions, especially over the last decade, which have since been verified as
being examples of known families.
The result of Coulter and Henderson \cite{CH-2008-cpas} linking planar DO
polynomials with
commutative semifields does complicate the history a little, but we are
somewhat fortunate in the sense that almost no planar DO polynomials known at
the time of the discovery corresponded with known families of commutative
semifields. That said, given the equivalence, any families are named after whatever came first, whether that be a semifield or a planar function. 

Finally, we note that in her Ph.D. thesis \cite{K-2009-csooo}, 
Kosick determines planar DO polynomial representatives for all of the
commutative semifield families discovered before 2009. Specifically,
planar DO polynomial representatives are given for the Dickson, Cohen-Ganley,
Ganley families, as well as the sporadic example of Pentilla and Williams.
See also the unpublished paper \cite{CHHKXZ-2007-ppacs} which, among other things, gives a family of planar DO binomials that represent a subset of the Dickson semifields.

\begin{enumerate}[label=\arabic*]
\setcounter{enumi}{1892}
\item {\bf FF} -- Finite fields
    \begin{itemize}
        \item $\SF=(\F_p^n,+,\star )$ given by
        $$x\star y=xy\text{ for all }x,y\in\F_p^n$$
        \item Planar DO representative: 
        $$F(x)=x^2.$$
        \item $(|N(\SF)|,|N_m(\SF)|)=(p^n,p^n)$.
    \end{itemize}
The monomial $x^2$ is planar over any field, finite or infinite, which is not
of characteristic 2.
It is easily seen to be equivalent to the field on which it is defined.
We give \arabic{enumi} as the year as it was then that Moore first established 
the uniqueness of finite fields of arbitrary order, thereby completing their
classification, see \cite{M-1893-adiso, M-1896-adiso}.

\setcounter{enumi}{1905}
\item {\bf D} -- The commutative semifields of Dickson
    \begin{itemize}
        \item $\SF=(\mathbb{F}_{p^m}^{2},+,\star )$ given by
        $$(a, b) \star (c, d) = (a c +\alpha (bd)^{p^i}, ad + bc)\text{ for all }a,b,c,d\in\F_{p^m}.$$
        Here, $\alpha\in\F_{p^m}$ is a nonsquare and $0<i\leq \lfloor \frac{m}{2}\rfloor$. 
        \item Different choices of $i$ lead to non-isotopic classes, while the two integers $i$ and $i'=m-i$ yield isotopic semifields, see Burmester \cite{B-1962-otcna}.
        \item For fixed i, all non-square $\alpha$ produce isotopic semifields, see \cite{D-1906-oclai}.
        \item There are multiple choices for the planar DO representatives for this class. Perhaps the best general forms known prior to this article are those given by Kosick \cite{K-2009-csooo} (see Lemma \ref{Dicksonreps} below). Also, a notable one is in binomial form \cite{CHHKXZ-2007-ppacs} but has many restrictions (see Lemma \ref{lem:BinomDicksonreps} below). In this article, we shall establish a class of quadrinomials that can be used as the representatives, see Theorem \ref{quadreps}.
        \item $(|N(\SF)|,|N_m(\SF)|)=(p^{\gcd(i,m)},p^{m})$.
    \end{itemize}
Dickson produced his family of commutative semifields in 1906 in 
\cite{D-1906-oclai}.
(Note that a typo in Dembowski's book \cite{bD-1968-fg}, page 241, inadvertently cites his nearfield paper \cite{D-1905-ofa} instead.)
He was motivated by his work \cite{D-1905-ofa} from the previous year in which he was the first to construct nearfields (a division ring where only one of the two distributive laws holds), describing an infinite family, along with a further 7 sporadic examples. That he had described all possible nearfields was only confirmed in 1935 by Zassenhaus \cite{Z-1935-ubef}.
It is, perhaps, worth mentioning that in 1907 Veblen and Wedderburn \cite{VW-1907-ndanp} used Dickson's nearfields to construct the first known examples of non-Desarguesian projective planes, and that their construction could
also have used Dickson's commutative semifields.

\setcounter{enumi}{1951}
\item {\bf A} -- Albert's twisted fields
    \begin{itemize}
        \item $\SF=(\F_p^n,+,\star )$ given by
        $$x\star y=x^{p^i}y+xy^{p^i}\text{ for all }x,y\in\F_p^n$$
        \item Planar DO representative: 
        $$F(x) = x^{p^{i}+1}$$
        where $1\leq i\leq \lfloor \frac{n}{2}\rfloor$ with $\frac{n}{\gcd(n,i)}$ odd.
        \item Different choices of $i$ lead to non-isotopic classes, see \cite{A-1952-onda}.
        \item $(|N(\SF)|,|N_m(\SF)|)=(p^{\gcd(n,i)},p^{\gcd(n,i)})$.
    \end{itemize}

Albert's twisted fields were first published in \cite{A-1952-onda}, and
subsequently generalised and extensively studied by him in the papers
\cite{A-1960-fdaaf, A-1961-gtf, A-1961-ifgtf, A-1963-otcga}.
The form of the equivalent planar monomials, the DO monomials, were first
described by Dembowski and Ostrom \cite{DO-1968-poowc}, but with an erroneous
condition, subsequently corrected by Coulter and Matthews in
\cite{CM-1997-pfapo}. Note that both Albert's original construction and the
DO monomial examples include the finite field case as an example. 

\setcounter{enumi}{1981}
\item {\bf CG} -- The commutative semifields of Cohen and Ganley
    \begin{itemize}
        \item $\SF=(\mathbb{F}_{3^m}^{2},+,\star )$ given by
        $$(a, b) \star (c, d) = (a c +\alpha (bd)+\alpha^3 (bd)^9, ad + bc+\alpha(bd)^3) \text{ for all } a,b,c,d\in\F_{3^m}.$$
        Here, $m\geq 3$, and $\alpha \in \mathbb{F}_{3^m}$ is a nonsquare.
        \item All non-square $\alpha$ produce isotopic semifields, see \cite{CG-1982-cstdo}.
        
        \item Planar DO representative over $\mathbb{F}_{3^{2m}}$: 
            $$F(x) = L(t^{2}(x)) + \frac{1}{2}x^2.$$
        Here $m \geq 3$ odd, $t = x^{3^m}-x$, $\beta \in \F_{3^{2m}}\setminus\F_{3^m}$, $\alpha=t(\beta)$, and $L(x)=-x^9-\alpha x^3 + (1-\alpha^4)x$.
        
        \item Each member of this family is isotopic inequivalent to the Dickson (D) family.
        \item There are multiple choices for the planar DO representatives for this class. The planar DO representatives we use here were given by Kosick \cite{K-2009-csooo}. We remark that there is a typo in the DO representatives given by Kosick, but the proofs are correct for the forms we give.
        \item $(|N(\SF)|,|N_m(\SF)|)=(3,3^m)$.
    \end{itemize}

    They were first constructed by Cohen and Ganley \cite{CG-1982-cstdo}, in
    their study of a general construction method for commutative semifields two
    dimensional over their middle nuclei. 

\setcounter{enumi}{1981}
\item {\bf G} -- The commutative semifields of Ganley
    \begin{itemize}
        \item $\SF=(\mathbb{F}_{3^m}^{2},+,\star )$ given by
            $$(a, b) \star (c, d) = (ac - b^{9}d - bd^{9}, ad + bc + b^{3}d^{3})\text{ for all }a,b,c,d\in\F_{3^m}.$$
        Here, $m\geq 3$ odd.
        
        \item Planar DO representative over $\mathbb{F}_{3^{2m}}$: 
            $$F(x) = L(t^{2}(x)) + D(t(x)) + \frac{1}{2}x^2.$$
        Here $m \geq 3$ odd, $t = x^{3^m}-x$, $\beta \in \F_{3^{2m}}\setminus\F_{3^m}$, $\alpha=t(\beta)$, $L(x)=-\alpha^{-5}x^3+x$, and $D(x)=-\alpha^{-10}x^{10}$.
        
        \item There are multiple choices for the planar DO representatives for this class. The planar DO representatives we use here were given by Kosick \cite{K-2009-csooo}. As for the CG family, we remark that there is a typo in the DO representatives given by Kosick, but the proofs are correct for the forms we give.
        \item $(|N(\SF)|,|N_m(\SF)|)=(3,3)$.
    \end{itemize}
    
    Building on the work in \cite{CG-1982-cstdo}, Ganley used the concept of weak
    nuclei to produce the commutative semifields in \cite{G-1981-cwns}.

\setcounter{enumi}{1996}
\item {\bf CM} -- The planar monomials of Coulter and Matthews
    \begin{itemize}
        \item Planar DO representative over $\mathbb{F}_{3^n}$:
            $$F(x)=x^{(3^i + 1)/2}.$$
        Here, $i\geq 3$ is odd and $\gcd(n,i) = 1$.
    \end{itemize}

    These are the only non-DO examples of planar functions known, and as they are not DO polynomials, they do not produce commutative semifields. The Coulter-Matthews monomials were
    announced on an online forum in 1994 and presented in full at Fq3, the 3rd
    International Conference on Finite Fields and Applications, in
    Glasgow in July, 1995. Delays in publication meant that Coulter and Matthews'
    results only appeared in print in 1997 in \cite{CM-1997-pfapo}.
    In the same year, Helleseth and Sandberg \cite{HS-1997-spmwl} published a
    paper which included the Coulter-Matthews examples.
    However, precedent was confirmed publicly by Tor Helleseth at BFA 2018, the
    third International Workshop on Boolean Functions and Their Application, in
    Loen in June 2018. We are most grateful to Tor for the clarification. These are the only known non-quadratic planar. It is known that if a quadratic function is CCZ-equivalent to a monomial, then such monomial must be quadratic. To check CCZ-equivalence among monomials, it is enough to check that their exponents belong to the same cyclotomic set. Since EA coincides with CCZ for planar functions, we can classify the instances of CM separately from the quadratic planar functions.
    
\setcounter{enumi}{1996}
    \item {\bf TST} -- The Ten-Six-Two family of planar DO polynomials
        \begin{itemize}
            \item Planar DO representative over $\mathbb{F}_{3^n}$: 
                $$F(x) = x^{10} \pm x^6 - x^2 .$$
            Here, $n\geq 5$ odd.
            \item The two functions are inequivalent up to strong isotopism, see \cite{CH-2008-cpas}.
            \item $(|N(\SF)|,|N_m(\SF)|)=(3,3)$.
        \end{itemize}

    The Ten-Six-Two family of functions are so named for their form:
    $g_5(x^2,a)=x^{10}+ax^6-a^2x^2$, where $g_k(x,a)$ denotes the $k$th Dickson
    polynomial of the first kind. (For information regarding the Dickson polynomials
    of the first and second kind, see the monograph \cite{bLMT-1993-dp} of
    Lidl, Mullen and Turnwald.)
    Coulter and Matthews established the planarity of $g_5(x^2,1)$ in
    \cite{CM-1997-pfapo}. Ding and Yuan \cite{DY-2006-afosh} extended the proof to
    all $a\ne 0$ in 2007 as part of their disproof of a 70 year old
    conjecture on skew Hadamard difference sets.
    Coulter and Henderson \cite{CH-2008-cpas} settled the question of equivalence
    by showing there are two equivalence families, one for all square $a$ and one
    for all non-square $a$. Thus, one can use $a=\pm 1$ to describe the two families.

\setcounter{enumi}{1999}
    \item {\bf PW} -- A sporadic commutative semifield of order $3^{10}$
        \begin{itemize}
            \item $\SF=(\mathbb{F}_{3^5}^{2},+,\star )$ given by
                $$(a, b) \star (c, d) = ((ac + bd)^{9}, ad + bc + (bd)^{27})\text{ for all } a,b,c,d\in\F_{3^5}.$$
            \item Planar DO representatives over 
            \item It is not isotopic equivalent to a Dickson semifield or to a Cohen-Ganley semifield,  see \cite{PW-2000-oops}.
            \item $(|N(\SF)|,|N_m(\SF)|)=(3,3^5)$.
        \end{itemize}
    
    In their study of ovoids in the orthogonal space $O(5,3^5)$, Penttila and
    Williams \cite{PW-2000-oops} discovered this commutative semifield as part
    of a computational search. At the time of writing, this example
    is not part of any known infinite family.

\setcounter{enumi}{2007}
    \item {\bf ACW} -- A sporadic planar DO binomial over $\F_{3^5}$
        \begin{itemize}
            \item Planar DO representative over $\mathbb{F}_{3^5}$: 
                $$F(x)=x^{90} + x^2.$$
            \item $(|N(\SF)|,|N_m(\SF)|)=(3,3).$
        \end{itemize}

    At and Cohen \cite{AC-2008-antfa} found a new planar binomial over
    $\F_{3^5}$ in the course
    of testing a proposed method for establishing planarity.
    In July, 2009, in Dublin, at Fq9, the 9th International Conference on Finite
    Fields and Applications, Coulter and Kosick \cite{CK-2010-csooa}
    attributed the very same planar example to Guibiao Weng, who had informed them
    of the example in personal correspondence dating from March 2008 at the
    latest.
    It appears both At and Cohen, and Weng discovered this
    binomial independently at around the same time.

\setcounter{enumi}{2007}
    \item {\bf BH} -- The planar DO polynomials of Budaghyan and Helleseth
        \begin{itemize}
            \item Planar DO representative over $\F_{p^{2m}}$:
            $$F(x)=x^{p^m +1} +\omega \Tr_m^{2m}(\beta x^{p^s +1}).$$ 
            Here, $\omega \in \mathbb{F}_{p^{2m}}\setminus\F_{p^m}$, $\beta\in\F_{p^{2m}}$ is a non-square and $0<s<m$ with $\nu_2(s)\neq \nu_2(m)$ where, for any positive integer $k$, $\nu_2(k)$ is the non-negative integer such that $\frac{k}{2^{\nu_2(k)}}$ is odd.
            \item For fixed $s$, any choice of $\omega$ and $\beta$ produces functions in the same strong isotopy class, see Bierbrauer \cite{B-2011-csfpm} or Feng and Li \cite{FL-2018-otico}.
            \item For any $m>1$ odd, each member splits into two isotopic classes (see Remark \ref{rem:BH} below).
            \item $(|N(\SF)|,|N_m(\SF)|)=(p^{\gcd(m,s)},p^{2\gcd(m,s)})$ \cite{MP-2012-otnoa}.
        \end{itemize}

    Budaghyan and Helleseth \cite{BH-2008-npnmo} established a family of planar DO polynomials with many terms. These produced the first new infinite family of
    commutative semifields without a restriction on the characteristic to be discovered in over 50 years. For convenience, here we present the form defined by Bierbrauer in \cite{B-2011-csfpm,B-2010-nspaa} (see \cite[Theorem 7]{ZP-2013-anfos}) because it is shorter.
    Many families of planar DO polynomials
    discovered since have been shown to be equivalent to this family such as the Lunardon-Marino-Polverino-Trombetti-Bierbrauer (LMPTB) family \cite{B-2011-csfpm}, the Zha-Wang (ZW) family \cite{ZW-2009-nfopn}, and a family by Bierbrauer\cite{BBFMP-2018-afosi}. Even some parts of the ZP family, see below, turns out to be equivalent to this family. Over $\F_{3^4}$, all the instances of the BH family do not split. In this paper, we show that all the instances of the BH family over $\F_{3^8}$ that have nuclei $(|N(\SF)|,|N_m(\SF)|)=(3,3^2)$ split and the ones that have nuclei $(|N(\SF)|,|N_m(\SF)|)=(3^2,3^4)$ do not split.

\setcounter{enumi}{2008}
    \item {\bf ZKW} -- The planar DO binomials of Zha, Kyureghyan and Wang
        \begin{itemize}
            \item Planar DO representative over $\F_{p^{n}}$:
            $$F(x)=x^{p^{s} + 1} - \alpha^{p^{k} - 1}x^{p^k + p^{2k + s}}.$$ 
            Here $n=3k$, $\gcd(3, k) = 1$, $s$ positive integer such that $k \equiv s \pmod3$, $\frac{n}{\gcd(s, n)}$ is odd and $\alpha \in \mathbb{F}_{p^{3k}}$ is primitive.
            \item $(|N(\SF)|,|N_m(\SF)|)=(p^{\gcd(s,k)},p^{\gcd(s,k)})$ \cite{MP-2012-otnoa}.
        \end{itemize}

    A family of planar DO binomials was given in 2009 by
    Zha, Kyureghyan and Wang \cite{ZKW-2009-pnlba}. This family was obtained as a generalization of a family of APN binomials presented by Budaghyan, Carlet, and Leander \cite{BCL-2008-tcoqa}. They also showed that the
    family contains planar functions that are not monomials.

\setcounter{enumi}{2009}
    \item {\bf B} -- The planar DO binomials of Bierbrauer
        \begin{itemize}
            \item Planar DO representative over $\F_{p^{n}}$:
                $$F(x)=x^{p^{s} + 1} - \alpha^{p^k-1} x^{p^{3k} + p^{k + s}}.$$
            Here $n=4k$, such that $\frac{2k}{\gcd(2k,s)}$ is odd, $p^s=p^k=1\pmod{4}$, and $\alpha \in \mathbb{F}_{p^{3k}}$ is primitive.
            \item $(|N(\SF)|,|N_m(\SF)|)=(p^{\gcd(s,k)},p^{2\gcd(s,k)})$ \cite{MP-2012-otnoa}.
        \end{itemize}

    This family of planar DO binomials was established by Bierbrauer
    \cite{B-2010-nspaa}. These binomials have a remarkably similar structure to those discovered in \cite{ZKW-2009-pnlba}, but have distinct dimensions. They were also obtained as a generalization of the same family of APN binomials \cite{BCL-2008-tcoqa}.

\setcounter{enumi}{2009}
    \item {\bf CK} -- Two sporadic planar DO polynomials over fields of order $3^5$ and $5^5$
        \begin{itemize}
            \item Planar DO representatives over $\F_{p^5}$ with $p=3,5$:
                $$F(x) = L(t^{2}(x)) + D(t(x)) + \frac{1}{2}x^2.$$
            \item The first with $p=3$, $t(x) = x^3 - x$, $L(x) = -x^3$ and $D(x)=-x^{36}+x^{30}+x^{28}+x^4$ (or $L(x) = -x$ and $D(x) =-x^{36}+x^{28}+x^{12}+x^4$).
            \item $(|N(\SF)|,|N_m(\SF)|)=(3,3)$.
            \item The second with $p=5$, $t(x) = x^5 - x$, $L(x) = x^{5^3} + x^{5^2} + 2x^5 + 3x$ and $D(x)=0$ (or $L(x) = 2x^{5^2} + x^5$ and $D(x) = 2x^{5^3 + 5} + 2x^{5^2 + 1}$).
            \item $(|N(\SF)|,|N_m(\SF)|)=(5,5)$.
        \end{itemize}

    Using the form described in \cite{CHK-2007-ppfcs}, Coulter and Kosick 
    \cite{CK-2010-csooa} conducted an exhaustive search for families of planar DO
    polynomials that had representatives with coefficients in the prime subfield.
    This search was carried out for fields of order $3^5$ and $5^5$, and two new
    examples were discovered, one for each order. 
    Both remain outside of any known infinite family.

\setcounter{enumi}{2012}
    \item {\bf ZP} -- The commutative semifields of Zhou and Pott
        \begin{itemize}
            \item $\SF = (\mathbb{F}_{p^m}^{2},+,\star)$ given by
                $$(a, b) \star (c, d) = (a\circ_k c +\alpha(b\circ_k d)^{p^i}, ad + bc) \text{ for all } a,b,c,d\in\F_{p^m}^{2}.$$
            Here $0\leq k,i\leq \lfloor \frac{m}{2}\rfloor $ with $(i,k)\neq (0,0)$, $\frac{m}{\gcd(m,k)}$ odd, $x\circ_k y=x^{p^k}y+xy^{p^k}$ and $\alpha \in \mathbb{F}_{p^m}$ is a nonsquare.
            
            \item By fixing $i$ and $k$ as above, any choice of $\alpha\in\F_{p^m}$ non-square leads to a semifield that fall into the same strong isotopic class \cite{ZP-2013-anfos}.
            \item The D family coincides with ZP with parameters $k=0$ and $0< i\leq \lfloor \frac{m}{2}\rfloor $. The BH family over the finite field $\F_{p^m}$ such that $m$ is odd and $-1$ is a square in $\F_{p^m}$ coincides with ZP with parameters $i=0$ and $0< k\leq \lfloor \frac{m}{2}\rfloor $ \cite[Theorem 5]{ZP-2013-anfos}.
            \item  Let $0\leq i_2,i_2\leq \lfloor \frac{m}{2}\rfloor $ and $0< k_1,k_2\leq \lfloor \frac{m}{2}\rfloor $ with $(i_1,k_1)\neq (i_2,k_2)$ then the two semifileds defined by $(i_1,k_1)$ and $(i_2,k_2)$ are not isotopic. Moreover if $i=0$ and $0< k\leq \lfloor \frac{m}{2}\rfloor $, then the isotopic class of each semifiled defined by $(0,k)$ contains exactly two strong isotopic classes \cite[Theorem 6]{ZP-2013-anfos}.
            \item $(|N(\SF)|,|N_m(\SF)|)=\begin{cases}
                (p^{\gcd(m,k)},p^{2\gcd(m,k)})&i=0\\
                (p^{\gcd(m,k,i)},p^{\gcd(m,k)})&i>0.
            \end{cases}$ 
        \end{itemize}

    By cleverly replacing the field multiplication with twisted field
    multiplication in the general form of commutative semifields studied by Cohen
    and Ganley \cite{CG-1982-cstdo},
    itself a form based on Dickson's original construction \cite{D-1906-oclai},
    Zhou and Pott \cite{ZP-2013-anfos} produced a new general family of
    commutative semifields. This construction method generated many new
    inequivalent families.

\setcounter{enumi}{2022}
    \item {\bf GK} -- The commutative semifields of G\"olo\u glu and K\"olsch
        \begin{itemize}
            \item Planar DO representative over $\F_{p^m}^2$:
                $$F(x,y) = (x^{p^k+1}+\alpha y^{p^{k}+1}, x^{p^{k+m/2}}y+\beta \alpha^{-1}xy^{p^{k+m/2}}).$$
            Here $m$ is even and not a power of $2$, $\alpha\in\F_{p^m}$ is a nonsquare, $\beta\in\F_{p^m}$ is not a power of $p^{m/2}+1$ and $\frac{m}{\gcd(m,k)}$ is odd.
            \item $(|N(\SF)|,|N_m(\SF)|)=(p^{\gcd(k,m)/2},p^{\gcd(k,m)})$.
        \end{itemize}
    
    G\"olo\u glu and K\"olsch dramatically changed the number of known
    inequivalent commutative semifields in a landmark paper
    \cite{GK-2023-aebot}.
    They produced a new construction method which allowed them to 
    combine some of the previously known commutative semifields into new semifields.
    Through their construction method they were able to show that the number of
    non-isotopic commutative semifields of order $p^n$ grows exponentially 
    with $n$. This confirmed a conjecture of Kantor, who had proved a
    corresponding result for commutative semifields of even order in 2003
    \cite{K-2003-csass}, but the result did not transfer to odd characteristic.
\end{enumerate}

\begin{remark}
In 2007, Coulter, Henderson and Kosick \cite{CHK-2007-ppfcs} developed a general form for planar DO polynomials that can be used to describe all planar DO families.
Using this form they published a purported example over $\F_{3^8}$ with middle nucleus of order 9 and nucleus of order 3. Unfortunately, the paper has numerous typographical errors and the original example, often called the CHK semifield, was wrong (it is not even planar). A replacement example was given on Coulter's personal website, but it was sourced from a file of Dickson isotopes by mistake. At the time \cite{CHK-2007-ppfcs} was published, Coulter, Henderson and Kosick knew of 4 potentially inequivalent commutative semifields of order $3^8$ with $(|N(\SF )|, |N_m(\SF)|) = (3,9)$, but had been unable to prove the inequivalence of these among themselves. We have now confirmed that the four potential classes are all now covered by the BH class (indeed, Classes 8.5 and 8.6 in Table \ref{tab:InequivalentRepresentatives}). For these reasons, even though the 4 examples known to the authors of \cite{CHK-2007-ppfcs} do pre-date the BH class, we do not list the CHK example as part of the nomenclature.
\end{remark}

\begin{remark}\label{rem:BH}
Let $m>1$. In \cite{ZP-2013-anfos}, it is shown that ZP with $i=0$ and $k\neq 0$ coincides with BH if $m$ is odd and $-1$ is a square in $\F_{p^m}$ and that each of those instances split. Since $m$ is odd, $-1$ is a square in $\F_{p^m}$ if and only if $p=1\pmod{4}$. Marino and Polverino \cite{MP-2012-oiasi} prove that the BH family splits if $m$ is not a power of $2$, $\frac{m}{\gcd(m,s)}$ is odd and $p^{\gcd(m,s)}=3\pmod{4}$. Observe that if m is odd, then $p^{\gcd(m,s)}=p\pmod{4}$. If $m$ is even and $\frac{m}{\gcd(m,s)}$ is odd, then $\gcd(m,s)$ is even and $p^{\gcd(m,s)}=1\pmod{4}$. Combining the two results, we get that the BH family splits if $m$ is odd.
\end{remark}

\subsection{Some remarks on the classification of planar DO polynomials}
We record some known facts on the classification of planar functions, planar DO polynomials and commutative semifields.
\begin{itemize}
\item Planar functions over prime fields were classified in 1989 and 1990. All
are quadratic, and all are equivalent to $x^2$, which produces the
finite field.
This was established independently by three sets of authors:
Gluck \cite{G-1990-apapp}, Hiramine \cite{H-1989-acoap}, and 
R\' onyai and Sz\" onyi \cite{RS-1989-pfaff}.

\item Planar monomials have been classified over fields of order $p^n$ with
$n\in\{1,2,3,4\}$. Johnson \cite{J-1987-ppoot} proved the prime field case in
1987.
Order $p^2$ was completed by Coulter \cite{C-2006-tcopm}, and order $p^4$ by
Coulter and Lazebnik \cite{CL-2012-otcop}.
The order $p^3$ case was completed in 2022 by Bergman, Coulter
and Villa \cite{BCV-2022-cpmof}.

\item Knuth \cite{K-1965-fsapp} showed that commutative semifields of order
$p^2$ are necessarily isotopic to finite fields in 1965. In 1977,
Menichetti \cite{M-1977-oakcc} proved that any commutative semifield of 
dimension 3 over its nucleus is necessarily isotopic to an Albert twisted
field. Together these two results complete the classification of commutative
semifields of orders $p^2$ and $p^3$. 

\item Menichetti later proved in \cite{M-1996-daoaf} that if $n$ is prime and $q$ is sufficiently large, any commutative semifield of order $q^n$ with nucleus of order $q$ is equivalent to an Albert's twisted field.

\item Let $q$ be an odd prime power and $l$ a positive integer.
Blokhuis, Lavrauw and Ball \cite{BLB-2003-otcos} proved that if
$q\ge 4l^2 - 8l +2$, then
any commutative semifield of order $q^{2l}$ with 
$|N(\SF)|\ge q$, $|N_m(\SF)| \ge q^l$ is either equivalent to
the finite field or a Dickson semifield.
In particular, this shows that any commutative semifield of order $p^4$ with
$(|N(\SF)|, |N_m(\SF)|) = (p,p^2)$ is necessarily a Dickson semifield
or a finite field.
Note that the only remaining case to consider for order $p^4$ is with
$(|N(\SF)|, |N_m(\SF)|) = (p,p)$. At the time of writing, we know of
no examples.

\item Commutative semifields of orders $3^n$ have been classified for 
$n\le 5$. The results of Knuth and Menichetti mentioned above deal with
$n\le 3$. For order 81, Dickson [35] showed in 1906 ``by a tedious computation"
that the only commutative semifields were those given by his construction. To put this in context, it wasn't until 2008 that
Dempwolff \cite{D-2008-spoo} managed to enumerate all semifields of order 81.
Weng and Zeng \cite{WZ-2012-fropd} computed all commutative semifields
of order 243 in 2012.
\item For $m\in\{2,3,4,5\}$, every commutative semifield $\SF$ over $\F_{3^{2m}}$ with nuclei $(|N(\SF)|,|N_m(\SF)|)=(3,3^m)$ is either in Dickson, in Cohen-Ganley, or in Pentilla-Williams (if $m=5$) \cite{BLB-2003-otcos,MPT-2007-oflso,LR-2019-codrt}.
\item For $n$ up to $7$, all planar DO polynomials in $\F_{3^n}$ with coefficients in the prime field $\F_3$ were classified by Davidova and Kaleyski \cite{DK-2022-coadp}.
\end{itemize}


\section{Quadrinomials representing the Dickson family}\label{sec:DicksonQuadrinomials}
There are a number of planar DO representatives of the Dickson family known.
Some of these are particularly simple. 
For example, \cite[Theorem 4.2]{CHHKXZ-2007-ppacs} gives binomial
representatives for the specific Dickson semifields of dimension 2 over the middle nucleus and dimension 4 over the nucleus.
Kosick \cite{K-2009-csooo} gives also planar DO representatives for
any Dickson semifield. Here we present both for completeness.
\begin{lemma}[{\cite[Theorem 4.2]{CHHKXZ-2007-ppacs}}]\label{lem:BinomDicksonreps}
    Let $p$ be any odd prime and let $m\ge 2$ be an even positive integer such that $p^{m/2}=1\pmod 4$. Let $q=p^m$, let $\alpha$ be a primitive element of $\F_{q^2}$ and $e$ a positive integer. Then $F(x)=x^{q+1}+\alpha^{e(p^{m/2}-1)}x^{2p^{m/2}}$ is a planar DO polynomial over
$\F_{q^2}$ representing the Dickson semifield as described in class {\bf D}.
\end{lemma}
\begin{lemma}[{Kosick \cite[Theorem 3.1.2]{K-2009-csooo}}]\label{Dicksonreps}
Let $p$ be any odd prime, $q=p^m$ for some integer $m\ge 2$ and let
$0< i <m$. Let $t(x)=x^q-x$ and $L(x)=8^{-1} (x^{p^i} - x)$ be functions over $\F_{q^2}$.
Then $F(x) = L(t^2(x)) + 2^{-1} x^2$ is a planar DO polynomial over
$\F_{q^2}$ representing
the Dickson semifield as described in class {\bf D}.
\end{lemma}
We note that the proof given by Kosick in \cite{K-2009-csooo} actually proves
a slightly more general statement than the one we gave above.

The representatives given in Lemma \ref{Dicksonreps} have 6 terms. We now 
show that there are always quadrinomial representatives of Dickson semifields.
\begin{theorem} \label{quadreps}
Let $p$ be any odd prime, $q=p^m$ for some integer $m\ge 2$ and let
$0< i < m$.
If $\gcd(i,m)=\gcd(i,2m)$, then the polynomial
\begin{equation}\label{eq:QuadDickson}
x^2 + x^{q+1} - x^{(q+1)p^i} + x^{2qp^i}
\end{equation}
is a planar DO polynomial over $\F_{q^2}$ representing
the Dickson semifield as described in class {\bf D} above.
\end{theorem}
\begin{proof}
For simplicity of exposition, set $r=p^i$.
We shall first prove that the linearised polynomial
$M(x)=3 x - x^r + x^q + x^{qr}$ is a permutation polynomial over
$\F_{q^2}$ and then we show that $M(F(x))$, where $F(x)$ is as given in Lemma \ref{Dicksonreps},
is the quadrinomial \eqref{eq:QuadDickson} of our statement. That it is planar follows at once from the fact $M$ is a linear permutation.

Let us show that $M$ is a permutation. To do so, we need only show the only root of $M$ in $\F_{q^2}$ is $z=0$.
To that end, suppose $z\in\F_{q^2}$ satisfies $M(z)=0$.
We have 
\begin{align*}
0 &= 3z + z^q - z^r + z^{qr}\\
&= (z^q+z) + 2z + (z^q-z)^r,
\end{align*}
We conclude $2z + (z^q-z)^r\in\F_{q}$.
We therefore have
\begin{equation*}
2z^q + (z-z^q)^r = 2z + (z^q-z)^r,
\end{equation*}
and rearranging and simplifying yields
\begin{equation*}
0 = (z^q-z)^r - (z^q-z).
\end{equation*}
Now, all roots of $x^r -x$ in $\F_{q^2}$ lie in the field
$\F_{r}\cap\F_{q^2}$, which is equal to the field $\F_{r}\cap\F_{q}$ by
hypothesis since $\F_{r}\cap\F_{q^2}=\F_{p^{\gcd(i,2m)}}=\F_{p^{\gcd(i,m)}}=\F_{r}\cap\F_{q}$.
In particular, all roots of $x^r-x$ lie in $\F_{q}$, and so $z^q-z\in\F_{q}$.
This implies $z^q-z=z-z^q$ from which we conclude $z\in\F_{q}$.
Thus,
\begin{equation*}
0 = M(z) = 3z + z  -z^r + z^r = 4z,
\end{equation*}
proving $z=0$ is our only solution. We have proved $M$ is a permutation.

Let us show that $M(F(x))$, where $F(x)$ is as given in Lemma \ref{Dicksonreps},
is the quadrinomial \eqref{eq:QuadDickson} of our statement. Observe that \begin{align*}
    8\cdot M(L(x))&=3x^r-3x+x^{qr}-x^q-x^{r^2}+x^r+x^{qr^2}-x^{qr}\\
    &=x^{qr^2}-x^q-x^{r^2}+4x^r-3x.
\end{align*} and that $8\cdot M(L(t^2(x)))$ is equal to
\begin{align*}
    &(x^{2q}+x^2-2x^{q+1})^{qr^2}-(x^{2q}+x^2-2x^{q+1})^q-(x^{2q}+x^2-2x^{q+1})^{r^2}\\&+4(x^{2q}+x^2-2x^{q+1})^r-3(x^{2q}+x^2-2x^{q+1})\\
=&4(x^{2q}+x^2-2x^{q+1})^r-4(x^{2q}+x^2-2x^{q+1}).
\end{align*}
Therefore, we have that $2\cdot M(F(x))=2\cdot M(L(t^2(x)))+M(x^2)$ is equal to
\begin{align*}
    &(x^{2q}+x^2-2x^{q+1})^r-(x^{2q}+x^2-2x^{q+1})+\left(3x^2-x^{2r}+x^{2q}+x^{2qr}\right)\\
    =&2x^2+2x^{q+1}-2x^{r(q+1)}+2x^{2qr}.
\end{align*}
This concludes the proof.
\end{proof}
The condition $\gcd(i,m)=\gcd(i,2m)$, while a restriction, does not stop us
from obtaining all of the non-isotopic versions in Dickson's class.
This is because $i$ and $i'=m-i$ produce the same Dickson isotopes.
Let $i=2^hg$ and $m=2^l k$ with $g,k$ odd integers.
Note that $i'=2^a b$ with $a=\min(l,h)$ and for some odd integer $b$.
If $\gcd(i,m)\ne\gcd(i,2m)$, then $h>l$ and so $\gcd(i',m)=\gcd(i',2m)$.
Consequently, we can cover all of the non-isotopic semifields in the Dickson
class by choosing $i$ or $i'$ as appropriate.

\section{Computational results}

\subsection{The expansion search for new planar functions}\label{sec:expansion_search}

We describe the procedure we used to search for new instances of planar DO polynomials over $\F_{3^6}$, $\F_{3^7}$, $\F_{3^8}$, and $\F_{3^9}$.  It consists of taking a quadratic power function $F(x) = x^d$ in univariate representation, adding quadratic terms of the form $cx^{d'}$ progressively, and testing for planarity each time. This method is based on a similar method as the one introduced by Aleksandersen, Budaghyan, and Kaleyski \cite{ABK-2022-sfafb} for the search of APN functions over binary fields.
Then, we test if these functions are new up to CCZ-equivalence using the linear equivalence algorithm \cite[Algorithm 1]{IK-2022-darle}. We performed this computational investigation for $d\in \{2,10\}$ over $\F_{3^6}$, $d\in \{2,4,10,28\}$ over $\F_{3^7}$, $d\in \{2,4,10,28,82\}$ over $\F_{3^8}$, and $d\in \{2,4,10,28,82,244\}$ over $\F_{3^9}$. 

We carried out our searches on a server with a Dell Inc. Poweredge C4130 motherboard, Intel Xeon CPU E5-2690 v4 @ 2.60GHz, NVIDIA Tesla K80, and 512 GB DDR4 RAM @2300MHz. We give a complete report of the searches conducted and the necessary time in Table \ref{tab:searches_times}.

Let $\alpha$ be the primitive element of $\F_{3^n}$ over $\F_3$ chosen by MAGMA \cite{BCP-1997-tmasi}. In dimension 6, we find seven new CCZ-classes of planar functions. All of these can be obtained by expanding $x^2$ and $x^{10}$ by two or three terms with coefficients in $\F_{3^2}$. Two of the classes can be represented by trinomials, namely
\begin{equation*}
    f_1(x) = \alpha^{91}x^{30} + x^{10} + x^2
\end{equation*}
and
\begin{equation*}
    f_2(x) = \alpha^{91}x^{486} + x^{10} + x^2.
\end{equation*} 
Note also that $\alpha^{91}$ is primitive in $\F_{3^2}$, and so all the coefficients of these representations lie in the subfield $\F_{3^2}$.
The remaining five classes do not appear to have a trinomial representation but can be expressed using quadrinomials with coefficients in $\F_{3^2}$. These are:
\begin{equation*}
    \begin{split}
        f_{3}(x) & = \alpha^{182}x^{82} + {2}x^{10} + \alpha^{91}x^{6} + x^{2}, \\
        f_{4}(x) & = \alpha^{182}x^{82} + {2}x^{10} + \alpha^{273}x^{6} + x^{2}, \\
        f_{5}(x) & = \alpha^{91}x^{486} + \alpha^{182}x^{90} + {2}x^{10} + x^{2}, \\
        f_{6}(x) & = \alpha^{273}x^{486} + \alpha^{182}x^{90} + {2}x^{10} + x^{2}, \\
        f_{7}(x) & = \alpha^{273}x^{246} + \alpha^{182}x^{82} + \alpha^{91}x^{6} + x^{2}.
    \end{split}
\end{equation*}
In dimension $9$, we find two new CCZ-classes of planar functions with coefficients in $\F_{3}$. Both of these can be obtained by expanding $x^2$ by four elements.
They are:
\begin{equation*}
    \begin{split}
        f_1(x) & = x^{756} + x^{486} + x^{162} + x^6 + x^2, \\
        f_2(x) & = x^{486} + x^{162} + 2x^{84} + 2x^{18} + x^2.
    \end{split}
\end{equation*}

Throughout all of our computational investigations in $\F_{3^7}$, $\F_{3^8}$, and $F_{3^{10}}$ we were not able to find any previously unknown new planar functions (up to CCZ-equivalence). However, we found some nice representations with coefficients in $\F_3$ of some known planar functions. We recall that this is an important result because it allows us to compute the orbits more efficiently, for instance using \cite[Algorithm 5]{IK-2022-darle}.

\begin{table}
    
\end{table}

\subsection{On the known planar functions in characteristic 3}\label{secRepresentatives}
We now present a full account of the known planar functions over
$\F_{3^n}$ for $3\leq n\leq 11$, up to CCZ-equivalence. To start, we focus on the procedures used to complete this
categorisation and the accompanying data related to invariants. Our results can be found in Appendix \ref{app:main}. We use as reference the known families of functions listed in Section \ref{sec:families}, {the newly discovered sporadic functions over $\F_{3^6}$ and $\F_{3^9}$ described in Subsection \ref{sec:expansion_search}}, and the family of quadrinomials described in Section \ref{sec:DicksonQuadrinomials}. By doing this, we update the previous classification done in 2010 \cite{PZ-2010-scopf}. We list the representatives of the known CCZ-classes in these dimensions and their CCZ-invariants. Those are the order of the monomial automorphism group, the sizes of the nuclei of the associated semifield in the case of quadratic functions, and the multiset of the cardinalities of each orbit as described in Subsection \ref{sec:preliminaries:equivalence}. Moreover, we report a list of representatives of the right orbits of each CCZ-class representative we chose. Although this is not an invariant, knowledge of the orbit representatives can be used to speed up the equivalence test defined in \cite{IK-2022-darle} significantly, making it easier to run future tests for equivalence with the proposed representatives. We recall that the orbit representatives are not a linear invariant, but if two functions $F$ and $G$ are linearly equivalent such that $G=L_1 \circ F \circ L_2$, then $RO_{G,x} = RO_{F,L_2(x)}$.

Whenever we have some quadratic planar functions that are either sporadic or come from a family, we must verify if each leads to $1$ or $2$ strong isotopic classes by using the Coulter-Henderson Theorem \ref{thm:CH}. We recall the order of the nuclei are invariants up to isotopism, while the order of the automorphism group and the multiset of the cardinalities of each orbit are only invariants up to strong isotopism. 

The classification of known planar functions over $\F_{3^n}$ with $n\geq 7$ odd does not require an equivalence test because the known functions are either power functions for which we can use cyclotomic equivalence instead of CCZ-or belong to the two TST instances which are known to be CCZ-inequivalent to any power function. However, computing the invariants in dimension $n=7$ and $n=9$ was important for the expansion search. The cases $n=8$ and $n=10$ are still feasible because we can use all the theoretical results available to reduce the number of tests. The case $n=12$ is left for future work both because of limitations of computational power and because the number of representatives is too high.

In order to compute the nuclei efficiently, we use the following procedure. Let $\SF=(\F_q,+,*)$ be a commutative semifield, then the function that maps $(x,y,z)\in\F_q$ to $(x * y)*z$ has a unique polynomial representation $f(x,y,z)$ where all the nonzero monomial in the representation are of the form $x^ay^bz^c$ where $0\leq a,b,c\leq q-1$. Then $N(\SF)$ is equal to {the set of all $\alpha\in\F_q$} such that $f(\alpha,x,y)=f(x,y,\alpha)$ since $N(\SF)=N_{\ell}(\SF)$. Moreover, the set $N_{m}(\SF)$ is equal to the set of $\alpha\in\F_q$ such that $f(x,\alpha,y)=f(y,\alpha,x)$. This procedure is efficient if constructing $f(x,y,z)$ is not computationally expensive and this can be the case if $\SF$ is constructed starting from a planar function $F$ with a sufficiently sparse polynomial representation.

The automorphism group of the associated linear code from \cite{PZ-2010-scopf} is computed in a straightforward way using the Magma algebra system. Unfortunately, this is only possible for $\F_{3^n}$ with $n \le 6$; for higher dimensions, the memory needed to perform the computation becomes prohibitive. We also note that computing the automorphism group of a DO planar function over $\F_{3^6}$ equivalent to a power function can be more computationally expensive than checking for equivalence with \cite{IK-2022-darle}. So we took this into account throughout our computations {by checking first if a function is equivalent to a power function and only if it is not, we compute the order of its automorphism group}.

Linear equivalence tests were performed using \cite[Algorithm 1]{IK-2022-darle}. 
Once a classification was compiled, it was possible to compute the linear invariant using the right orbits, as well as the orbit representatives, using \cite[Algorithm 5]{IK-2022-darle}. {We made an effort in searching for the class representatives that have polynomial representation with coefficients in the smallest possible subfield, since this significantly speeds up the computation of the orbits.}

\section*{Acknowledgments}
The results of this paper are partially in the master thesis of Alise Haukenes \cite{H-2022-cacsf}, with supervisors Lilya Budaghyan and Nikolay Kaleyski. The research of Lilya Budaghyan and Enrico Piccione is supported by the Norwegian Research Council. R.S. Coulter's research was partially supported by a bequest from the Estate of Francisco Javier ``Pancho" Sayas.

\printbibliography

\newpage
\appendix
\section{Appendix}\label{app:main}
\subsection{Functions}\label{app:classification}

\renewcommand{\arraystretch}{1.2}
\setlength{\tabcolsep}{5pt}
\begin{xltabular}{\linewidth}{|c|c|X|c|c|}
    \caption{CCZ-inequivalent planar functions over $\F_{3^n}$, $n=2,\ldots,8$}
    \label{tab:InequivalentRepresentatives}
    \\
    \hline
    \textbf{n} & $\mathbf{N^{O}}$ & \textbf{Representative} & \textbf{Family} & Splits \\
    \hline
    \endhead
    \hline
    \endfoot
    
    \hline
    2  & 2.1 & $x^2$ &  FF & No \\
    \hline
    \multirow{2}{*}{3}  & 3.1 & $x^2$ &  FF & No \\
    & 3.2 & $x^4$ &  A & No \\
    \hline
    \multirow{3}{*}{4}  & 4.1 & $x^2$ &  FF & No \\
    & 4.2 & $x^{14}$ &  CM & NA \\
    & 4.3 & $x^{36} + 2x^{10} + 2x^4$ &  \Centerstack{BH/D/ZP} & No \\
    \hline
    \multirow{8}{*}{5}  & 5.1 & $x^2$ &  FF & No \\
    & 5.2 & $x^4$ &  A & No \\
    & 5.3 & $x^{10}$ &  A & No \\
    & 5.4 & $x^{14}$ &  CM & NA \\
    & 5.5 & $x^{10} + x^6 + 2x^2$ &  TST & No \\
    & 5.6 & $x^{10} + 2x^6 + 2x^2$ &  TST & No \\
    & 5.7 & $x^{90} + x^2$ &  ACW & No \\
    & 5.8 & $x^{162} + x^{108} + 2x^{84} + x^2$ &  CK & No \\
    \hline
    \multirow{24}{*}{6}  & 6.1 & $x^2$ &  FF & No \\
    & 6.2 & $x^{10}$ & A & No \\
    & 6.3 & $x^{14}$ & CM & NA \\
    & 6.4 & $x^{122}$ & CM & NA \\
    & 6.5 & $x^{162} + 2x^{84} + x^{28} + x^2$ &  D/ZP & No \\
    & 6.6 & $\alpha^{455}x^{270} + x^{28} + \alpha^{273}x^{10}$ & BH/ZP & 6.7 \\
    & 6.7 & $2x^{270} + x^{246} + 2x^{90} + x^{82} + x^{54} + 2x^{30} + x^{10} + x^2$ & BH/ZP & 6.6 \\
    & 6.8 & $x^{270} + 2x^{244} + \alpha^{449}x^{162} + \alpha^{449}x^{84} + \alpha^{534}x^{54} + 2x^{36} + \alpha^{534}x^{28} + x^{10} + \alpha^{449}x^6 + \alpha^{279}x^2$ & G & No \\
    & 6.9 & $x^{486} + x^{252} + \alpha^{561}x^{162} + \alpha^{561}x^{84} + \alpha^{183}x^{54} + \alpha^{183}x^{28} +x^{18} + \alpha^{561}x^6 + \alpha^{209}x^2$ & CG & No \\
    & 6.10 & $x^{162} + 2x^{108} + 2x^{90} + x^{82} + 2x^{10} + x^{4} + x^2$ & ZP & No \\
    & 6.11 & $\alpha^{91}x^{30} + x^{10} + x^2$ &  This work & No \\
    & 6.12 & $\alpha^{91}x^{486} + x^{10} + x^2 $ &  This work & No \\
    & 6.13 & $\alpha^{182}x^{82} + 2x^{10} + \alpha^{91}x^6 + x^2$ &  This work & 6.14 \\
    & 6.14 & $\alpha^{182}x^{82} + 2x^{10} + \alpha^{273}x^6 + x^2 $ &  This work & 6.13 \\
    & 6.15 & $\alpha^{91}x^{486} + \alpha^{182}x^{90} + 2x^{10} + x^2 $ &  This work & 6.16 \\
    & 6.16 & $\alpha^{273}x^{486} + \alpha^{182}x^{90} + 2x^{10} + x^2$ &  This work & 6.15 \\
 
    & 6.17 & $\alpha^{273}x^{246} + \alpha^{182}x^{82} + \alpha^{91}x^6 + x^2$ &  This work & No \\
    \hline
    \multirow{8}{*}{7} & 7.1 & $x^2$ &  FF & No \\
    & 7.2 & $x^4$ &  A & No \\
    & 7.3 & $x^{10}$ &  A & No \\
    & 7.4 & $x^{28}$ &  A & No \\
    & 7.5 & $x^{14}$ &  CM & NA \\
    & 7.6 & $x^{122}$ &  CM & NA \\
    & 7.7 & $x^{10} + x^6 + 2x^2$ &  TST & No \\
    & 7.8 & $x^{10} + 2x^6 + 2x^2$ &  TST & No \\
    \hline

    \multirow{14}{*}{8} & 8.1 & $x^2$ &  FF & No \\
    & 8.2 & $x^{14}$ &  CM & NA \\
    & 8.3 & $x^{122}$ &  CM & NA \\
    & 8.4 & $x^{1094}$ &  CM & NA \\
    & 8.5 & $x^{244} + 2x^{84} + 2x^{82}$ & BH & 8.6 \\
    & 8.6 & $x^{324} + x^{82} + 2x^4$ & BH & 8.5 \\
    & 8.7 & $x^{1458} + 2x^{738} + x^{82} + x^2$ &  \Centerstack{B/BH/D/ZP} & No \\
    & 8.8 & $x^{486} + 2x^{246} + x^{82} + x^2$ &  \Centerstack{D/ZP} & No \\
    & 8.9 & $\alpha^{3608}x^{1458} + \alpha^{3608}x^{738} + \alpha^{3810}x^{486} + \alpha^{3810}x^{246} + \alpha^{3413}x^{162} +\alpha^{3413}x^{82} + \alpha^{3608}x^{18} + \alpha^{3810}x^6 + \alpha^{2565}x^2$ & CG  & 8.10 \\
    & 8.10 & $\alpha^{164}x^{1458} + \alpha^{164}x^{738} + \alpha^{950}x^{486} + \alpha^{950}x^{246} + \alpha^{616}x^{162} +\alpha^{616}x^{82} + \alpha^{164}x^{18} + \alpha^{950}x^6 + \alpha^{6297}x^2$ & CG & 8.9 \\
    \hline
    \multirow{11}{*}{9}
    & 9.1 & $x^2$ &  FF & No \\
    & 9.2 & $x^4$ &  A & No \\
    & 9.3 & $x^{10}$ &  A & No \\
    & 9.4 & $x^{28}$ &  A & No \\
    & 9.5 & $x^{82}$ &  A & No \\
    & 9.6 & $x^{122}$ &  CM & NA \\
    & 9.7 & $x^{1094}$ &  CM & NA \\
    & 9.8 & $x^{10} + x^6 + 2x^2$ &  TST & No \\
    & 9.9 & $x^{10} + 2x^6 + 2x^2$ &  TST & No \\
    & 9.10 & $x^{486} + x^{162} + 2x^{84} + 2x^{18} + x^2$ & This work & No \\
    & 9.11 & $x^{756} + x^{486} + x^{162} + x^6 + x^2$ & This work & No \\
    \hline
    \pagebreak
    \multirow{28}{*}{10}
    & 10.1 & $x^2$ &  FF & No \\
    & 10.2 & $x^{10}$ &  A & No \\
    & 10.3 & $x^{82}$ &  A & No \\
    & 10.4 & $x^{14}$ &  CM & NA \\
    & 10.5 & $x^{1094}$ &  CM & NA \\
    & 10.6 & $x^{9842}$ &  CM & NA \\
    & 10.7 & $2 x^{4374} + 2 x^{2196} + \alpha^{7686} x^{1458} + \alpha^{7686} x^{732} + \alpha^{244} x^{486} + \alpha^{244} x^{244} + 2 x^{18} +\alpha^{7686} x^{6} + \alpha^{1220} x^2$ &  CG & No \\
    & 10.8 & $x^{1458} + 2x^{732} + x^{244} + x^2$  &  D/ZP & No \\
    & 10.9 & $x^{13122} + 2x^{6588} + x^{244} + x^2$  &  D/ZP & No \\
    & 10.10 & $\alpha^{44286} x^{13122} + \alpha^{44286}x^{6588} + x^{4374} + x^{2196} + x^{486} + x^{244} + \alpha^{44286} x^{54} + x^{18}$  & PW & No \\
    & 10.11 & $x^{2430} + 2x^{2188} + \alpha^{14762}x^{1458} + \alpha^{14762}x^{732} + x^{486} + 2x^{252} + x^{244} + x^{10} + \alpha^{14762} x^6$ &  G & No \\
    & 10.12 & $2x^{2916} + x^{738} + x^{12}$ & ZP & No \\ 
    & 10.13 & $x^{21870} + x^{19692} + x^{2430} + 2x^{252}$ & ZP & No \\ 
    & 10.14 & $x^{21870} + x^{19692} + 2x^{2268}$ & ZP & No \\ 
    & 10.15 & $x^{21870} + 2x^{19692} + x^{2430} + 2x^{2268} + x^{2188} + 2x^{486} + x^{252} + x^{90} + x^{10} + x^2$ & ZP & No \\ 
    & 10.16 & $x^{2430} + x^{244} + \alpha^{44286}x^{10}$ & BH & 10.17 \\
    & 10.17 & $x^{6562} + x^{486} + 2x^{270} + x^2$ & BH & 10.16 \\
    & 10.18 & $x^{19926} + x^{244} + \alpha^{44286}x^{82}$ & BH & 10.19 \\
    & 10.19 & $2x^{19926} + x^{486} + x^{82} + x^2$ & BH & 10.18\\ 
    & 10.20 & $2x^{19926} + 2x^{486} + 2x^{82} + x^2$ & ZP & 10.21 \\
    & 10.21 & $x^{19764} + x^{2190} + a^{242}x^{738}$ & ZP & 10.20 \\
    & 10.22 & $2x^{2430} + 2x^{486} + 2x^{10} + x^2$ & ZP & 10.23 \\
    & 10.23 & $x^{19764} + a^{242}x^{6562} + x^{732} + x^{270}$ & ZP & 10.22 \\
    \hline
    \pagebreak
    \multirow{12}{*}{11}
    & 11.1 & $x^2$ &  FF & No \\
    & 11.2 & $x^4$ &  A & No \\
    & 11.3 & $x^{10}$ &  A & No \\
    & 11.4 & $x^{28}$ &  A & No \\
    & 11.5 & $x^{82}$ &  A & No \\
    & 11.6 & $x^{244}$ &  A & No \\
    & 11.7 & $x^{14}$ &  CM & NA \\
    & 11.8 & $x^{122}$ &  CM & NA \\
    & 11.9 & $x^{1094}$ &  CM & NA \\
    & 11.10 & $x^{9842}$ &  CM & NA \\
    & 11.11 & $x^{10} + x^6 + 2x^2$ &  TST & No \\
    & 11.12 & $x^{10} + 2x^6 + 2x^2$ &  TST & No \\
\end{xltabular}
\setlength{\tabcolsep}{6pt}
\renewcommand{\arraystretch}{1}

\subsection{Invariants}

\renewcommand{\arraystretch}{1.2}
\setlength{\tabcolsep}{5pt}
\begin{xltabular}{\linewidth}{|c|c|c|c|c|X|}
    \caption{Invariants for the classes in Table \ref{tab:InequivalentRepresentatives}}
    \label{tab:InvariantsNucleiCode}\\
    \hline
      \textbf{n} & $\mathbf{N^{O}}$ & $|N|$ & $|N_m|$ & Orbits & Aut. group order \\
      \hline
      \endhead
      \hline
      \endfoot
      \hline
      \multirow{2}{*}{3}
      & 3.1 & $3^3$ & $3^3$ & $\left\{*26*\right\}$ & $4212$ \\
      & 3.2 & $3$ & $3$  & $\left\{*26*\right\}$ & $4212$ \\
      \hline
      \multirow{3}{*}{4}
      & 4.1 & $3^4$ & $3^4$& $\left\{*80*\right\}$ & $51840$ \\
      & 4.2 & NA & NA & $\left\{*80*\right\}$ & $640$ \\
      & 4.3 & $3$ & $3^2$ & $\left\{*16,64*\right\}$ & $10368$ \\
      \hline
      \multirow{8}{*}{5}
      & 5.1 & $3^5$ & $3^5$  & $\left\{*242*\right\}$ & $588060$ \\
      & 5.2 & $3$ & $3$  & $\left\{*242*\right\}$ & $588060$ \\
      & 5.3 & $3$ & $3$  & $\left\{*242*\right\}$ & $588060$ \\
      & 5.4 & NA & NA  & $\left\{*242*\right\}$ & $2420$ \\
      & 5.5 & $3$ & $3$  & $\left\{* 2, 10^{24} *\right\}$ & $4860$ \\
      & 5.6 & $3$ & $3$  & $\left\{* 2, 10^{24} *\right\}$ & $4860$ \\
      & 5.7 & $3$ & $3$  & $\left\{* 22, 110^2 *\right\}$ & $53460$ \\
      & 5.8 & $3$ & $3$  & $\left\{* 2, 10^{24} *\right\}$ & $4860$ \\
      \hline
\pagebreak
      \multirow{17}{*}{6}
      & 6.1 & $3^6$ & $3^6$  & $\left\{*728*\right\}$ & $6368544$ \\
      & 6.2 & $3^2$ & $3^2$  & $\left\{*728*\right\}$ & $6368544$ \\
      & 6.3 & NA & NA  & $\left\{*728*\right\}$ & $8736$ \\
      & 6.4 & NA & NA  & $\left\{*728*\right\}$ & $8736$ \\
      & 6.5 & $3$ & $3^3$  & $\left\{* 26^2, 52, 156^4 *\right\}$ & $227448$ \\
      & 6.6 & $3$ & $3^2$  & $\left\{* 104, 312^2 *\right\}$ & $454896$ \\
      & 6.7 & $3$ & $3^2$  & $\left\{* 52^2, 312^2 *\right\}$ & $454896$ \\
      & 6.8 & $3$ & $3$  & $\left\{* 26^4, 78^8 *\right\}$ & $113724$ \\
      & 6.9 & $3$ & $3^3$  & $\left\{* 8^4, 24^{29} *\right\}$ & $34992$ \\
      & 6.10 & $3$ & $3$  & $\left\{* 26^2, 52, 156^4 *\right\}$ & $227448$ \\
      & 6.11 & $3$ & $3^2$  & $\left\{* 4^2, 12^{60} *\right\}$ & $17496$ \\
      & 6.12 & $3$ & $3^2$  & $\left\{* 4^2, 12^{60} *\right\}$ & $17496$ \\
      & 6.13 & $3$ & $3^2$  & $\left\{* 4^2, 12^{60} *\right\}$ & $17496$ \\
      & 6.14 & $3$ & $3^2$  & $\left\{* 4^2, 12^{60} *\right\}$ & $17496$ \\
      & 6.15 & $3$ & $3^2$  & $\left\{* 4^2, 12^{60} *\right\}$ & $17496$ \\
      & 6.16 & $3$ & $3^2$  & $\left\{* 4^2, 12^{60} *\right\}$ & $17496$ \\
      & 6.17 & $3$ & $3^2$  & $\left\{* 4^2, 12^{60} *\right\}$ & $17496$ \\
      \hline
    \multirow{8}{*}{7}
      & 7.1 & $3^7$ & $3^7$ & $\left\{*2186*\right\}$ & $-$ \\
      & 7.2 & $3$ & $3$ & $\left\{*2186*\right\}$ & $-$ \\
      & 7.3 & $3$ & $3$ & $\left\{*2186*\right\}$ & $-$ \\
      & 7.4 & $3$ & $3$ & $\left\{*2186*\right\}$ & $-$ \\
      & 7.5 & NA & NA & $\left\{*2186*\right\}$ & $-$ \\
      & 7.6 & NA & NA & $\left\{*2186*\right\}$ & $-$ \\
      & 7.7 & $3$ & $3$ & $\left\{*2, 14^{156}*\right\}$ & $-$ \\
      & 7.8 & $3$ & $3$ & $\left\{*2, 14^{156}*\right\}$ & $-$ \\
    \hline
      \multirow{10}{*}{8}
      & 8.1 & $3^8$ & $3^8$   & $\left\{* 6560 *\right\}$ & $-$ \\
      & 8.2 & NA & NA   & $\left\{* 6560 *\right\}$ & $-$ \\
      & 8.3 & NA & NA & $\left\{* 6560 *\right\}$ & $-$ \\
      & 8.4 & NA & NA & $\left\{* 6560 *\right\}$ & $-$ \\
      & 8.5 & $3$ & $3^2$   & $\left\{* 160, 1280^5 *\right\}$ & $-$ \\
      & 8.6 & $3$ & $3^2$   & $\left\{* 160, 1280^5 *\right\}$ & $-$ \\
      & 8.7 & $3^2$ & $3^4$ & $\left\{* 160, 1280^5 *\right\}$ & $-$ \\
      & 8.8 & $3$ & $3^4$   & $\left\{* 80^2, 640^{10} *\right\}$ & $-$ \\
      & 8.9 & $3$ & $3^4$   & $\left\{* 16^{410} *\right\}$ & $-$ \\
      & 8.10 & $3$ & $3^4$  & $\left\{* 16^{410} *\right\}$ & $-$ \\
      \hline
\pagebreak
      \multirow{11}{*}{9}
    & 9.1 & $3^9$ & $3^9$ & $\left\{*19682*\right\}$ & $-$ \\
    & 9.2 & $3$ & $3$ & $\left\{*19682*\right\}$ & $-$ \\
    & 9.3 & $3$ & $3$ & $\left\{*19682*\right\}$ & $-$ \\
    & 9.4 & $3^3$ & $3^3$ & $\left\{*19682*\right\}$ & $-$ \\
    & 9.5 & $3$ & $3$ & $\left\{*19682*\right\}$ & $-$ \\
    & 9.6 & NA & NA & $\left\{*19682*\right\}$ & $-$ \\
    & 9.7 & NA & NA & $\left\{*19682*\right\}$ & $-$ \\
    & 9.8 & $3$ & $3$ & $\left\{* 2, 6^4, 18^{1092} *\right\}$ & $-$ \\
    & 9.9 & $3$ & $3$ & $\left\{* 2, 6^4, 18^{1092} *\right\}$ & $-$ \\
    & 9.10 & $3$ & $3^3$ & $-$ & $-$ \\
    & 9.11 & $3$ & $3^3$ & $-$ & $-$ \\
    \hline
    \multirow{23}{*}{10}
    & 10.1 & $3^{10}$ & $3^{10}$ & $\left\{*59048*\right\}$ & $-$ \\
    & 10.2 & $3^{2}$ & $3^{2}$ & $\left\{*59048*\right\}$ & $-$ \\
    & 10.3 & $3^{2}$ & $3^{2}$ & $\left\{*59048*\right\}$ & $-$ \\
    & 10.4 & NA & NA & $\left\{*59048*\right\}$ & $-$ \\
    & 10.5 & NA & NA & $\left\{*59048*\right\}$ & $-$ \\
    & 10.6 & NA & NA & $\left\{*59048*\right\}$ & $-$ \\
    & 10.7 & $3$ & $3^5$ & $-$ & $-$ \\
    & 10.8 & $3$ & $3^5$ & $\left\{*242^2, 484, 2420^{24} *\right\}$ & $-$ \\
    & 10.9 & $3$ & $3^5$ & $\left\{*242^2, 484, 2420^{24} *\right\}$ & $-$ \\
    & 10.10 & $3$ & $3^5$ & $-$ & $-$ \\
    & 10.11 & $3$ & $3$ & $-$ & $-$ \\
    & 10.12 & $3$ & $3$ & $\left\{*242^2, 484, 2420^{24}*\right\}$ & $-$ \\
    & 10.13 & $3$ & $3$ & $\left\{* 242^2, 484, 2420^{24} *\right\}$ & $-$ \\
    & 10.14 & $3$ & $3$ & $\left\{*242^2, 484, 2420^{24}*\right\}$ & $-$ \\
    & 10.15 & $3$ & $3$ & $\left\{* 242^2, 484, 2420^{24} *\right\}$ & $-$ \\
    & 10.16 & $3$ & $3^2$ & $\left\{*968, 4840^{12}*\right\}$ & $-$ \\
    & 10.17 & $3$ & $3^2$ & $\left\{* 484^2, 4840^{12} *\right\}$ & $-$ \\
    & 10.18 & $3$ & $3^2$ & $\left\{*968, 4840^{12}*\right\}$ & $-$ \\
    & 10.19 & $3$ & $3^2$ & $\left\{* 484^2, 4840^{12} *\right\}$ & $-$ \\
    & 10.20 & $3$ & $3^2$ & $\left\{*484^2, 4840^{12}*\right\}$ & $-$ \\
    & 10.21 & $3$ & $3^2$ & $\left\{*968, 4840^{12}*\right\}$ & $-$ \\
    & 10.22 & $3$ & $3^2$ & $\left\{*484^2, 4840^{12}*\right\}$ & $-$ \\
    & 10.23 & $3$ & $3^2$ & $\left\{*968, 4840^{12}*\right\}$ & $-$ \\
    \hline
\pagebreak
    \multirow{12}{*}{11}
    & 11.1 & $3^{11}$ & $3^{11}$ & $\left\{*177146*\right\}$ & $-$ \\
    & 11.2 & $3$ & $3$ & $\left\{*177146*\right\}$ & $-$ \\
    & 11.3  & $3$ & $3$ & $\left\{*177146*\right\}$ & $-$ \\
    & 11.4 & $3$ & $3$ & $\left\{*177146*\right\}$ & $-$ \\
    & 11.5  & $3$ & $3$ & $\left\{*177146*\right\}$ & $-$ \\
    & 11.6  & $3$ & $3$ & $\left\{*177146*\right\}$ & $-$ \\
    & 11.7  & NA & NA & $\left\{*177146*\right\}$ & $-$ \\
    & 11.8 & NA & NA & $\left\{*177146*\right\}$ & $-$ \\
    & 11.9 & NA & NA & $\left\{*177146*\right\}$ & $-$ \\
    & 11.10 & NA & NA & $\left\{*177146*\right\}$ & $-$ \\
    & 11.11 & $3$ & $3$ & $-$ & $-$ \\
    & 11.12 & $3$ & $3$ & $-$ & $-$ \\
\end{xltabular}
\renewcommand{\arraystretch}{1}
\setlength{\tabcolsep}{6pt}


\subsection{Orbits}

\renewcommand{\arraystretch}{1.2}
\setlength{\tabcolsep}{5pt}
\begin{xltabular}{\linewidth}{|c|c|>{\centering\arraybackslash}X|}
    \caption{Right orbit representatives for the classes in Table \ref{tab:InequivalentRepresentatives}.}
    \label{tab:OrbitsRepresentatives}\\
    \hline
    \textbf{n} & $\mathbf{N^{O}}$ & Orbit representatives ($\alpha^i$) \\
    \hline
    \endhead
    \hline
    \endfoot
    \hline
    4 & 4.3 & $0, 1$ \\
    \hline
    \multirow{5}{*}{5} & 5.5 & 0, 1, 2, 4, 5, 7, 8, 10, 11, 13, 16, 17, 19, 20, 22, 25, 26, 31, 34, 35, 38, 40, 61, 67, 76 \\
    & 5.6 & same as 5.5 \\
    & 5.7 & $0,1,2$ \\
    & 5.8 & same as 5.5 \\
    \hline
    \multirow{14}{*}{6} 
    & 6.5 & 0, 1, 4, 7, 8, 11, 14 \\
    & 6.6 & 0, 1, 2 \\
    & 6.7 & 0, 1, 2, 7 \\
    & 6.8 & 0, 1, 2, 3, 4, 5, 6, 8, 10, 15, 17, 20 \\
    & 6.9 & 0, 1, 2, 3, 4, 5, 6, 7, 8, 9, 10, 12, 13, 16, 17, 19, 22, 23, 31, 34, 35, 36, 38, 39, 44, 45, 47, 48, 50, 54, 66, 72, 90 \\
    & 6.10 & 0, 1, 2, 6, 8, 13, 15 \\
    & 6.11 & Table \ref{tab:RepresentativesTooManyOrbits} \\
    & 6.12 & Table \ref{tab:RepresentativesTooManyOrbits} \\
    & 6.13 & Table \ref{tab:RepresentativesTooManyOrbits} \\
    & 6.14 & Table \ref{tab:RepresentativesTooManyOrbits} \\
    & 6.15 & Table \ref{tab:RepresentativesTooManyOrbits} \\
    & 6.16 & Table \ref{tab:RepresentativesTooManyOrbits} \\
    & 6.17 & Table \ref{tab:RepresentativesTooManyOrbits} \\
    \hline
\pagebreak
    \multirow{2}{*}{7}
    & 7.7 & Table \ref{tab:RepresentativesTooManyOrbits} \\
    & 7.8 & Table \ref{tab:RepresentativesTooManyOrbits} \\
    \hline
    \multirow{6}{*}{8}
    & 8.5 & 0, 1, 2, 4, 7, 8 \\
    & 8.6 & 0, 1, 2, 4, 7, 8 \\
    & 8.7 & 0, 1, 2, 4, 5, 7 \\
    & 8.8 & 0, 1, 2, 4, 10, 11, 13, 16, 17, 28, 35, 41 \\
    & 8.9 & Table \ref{tab:RepresentativesTooManyOrbits} \\
    & 8.10 & Table \ref{tab:RepresentativesTooManyOrbits} \\
    \hline
    \multirow{18}{*}{10}
    & 10.8  & 0, 1, 4, 5, 7, 10, 11, 14, 16, 19, 20, 22, 26, 31, 34, 38, 40, 41, 49, 55, 65, 76, 82, 91, 104, 122, 133 \\
    & 10.9  & 0, 1, 2, 4, 7, 8, 10, 13, 14, 16, 19, 20, 23, 25, 26, 32, 34, 38, 40, 41, 44, 55, 61, 86, 122, 125, 188  \\
    & 10.12 & 0, 1, 2, 4, 5, 7, 8, 10, 11, 13, 14, 16, 17, 19, 20, 22, 23, 31, 32, 34, 38, 40, 41, 43, 47, 61, 122 \\
    & 10.13 & Same as 10.9 \\
    & 10.14 & 0, 1, 2, 3, 4, 5, 6, 7, 8, 11, 16, 20, 21, 23, 28, 29, 31, 34, 38, 42, 43, 48, 53, 64, 79, 183, 192 \\
    & 10.15 & Same as 10.12 \\
    & 10.16 & 0, 1, 2, 3, 4, 5, 6, 8, 10, 11, 12, 13, 17 \\
    & 10.17 & 0, 1, 2, 4, 5, 7, 8, 10, 11, 16, 17, 19, 20, 61 \\
    & 10.18 & same as 10.16 \\
    & 10.19 & same as 10.17 \\
    & 10.20 & same as 10.17 \\
    & 10.21 & 0, 1, 2, 3, 4, 6, 7, 8, 10, 13, 42, 44, 51 \\
    & 10.22 & same as 10.17 \\
    & 10.23 & 0, 1, 2, 3, 4, 6, 7, 10, 12, 13, 14, 15, 19 \\
\end{xltabular}
\renewcommand{\arraystretch}{1}
\setlength{\tabcolsep}{6pt}

\renewcommand{\arraystretch}{1.2}
\setlength{\tabcolsep}{5pt}
\begin{xltabular}{\linewidth}{|c|X|}
    \caption{Right orbits representatives of the classes missing from Table \ref{tab:OrbitsRepresentatives}}
    \label{tab:RepresentativesTooManyOrbits}
    \\
    \hline
    $\mathbf{N^{O}}$ & Orbit representatives ($\alpha^i$) \\
    \hline
    \endhead
    \hline
    \endfoot
    \hline
    \multirow{4}{*}{\Centerstack{6.11 \\ to \\6.17}}& 0, 1, 2, 3, 4, 5, 6, 7, 8, 10, 11, 12, 13, 14, 15, 16, 17, 19, 20, 23, 24, 26, 28, 29, 30, 31, 32, 33, 35, 37, 38, 39, 40, 46, 47, 48, 49, 51, 53, 55, 56, 57, 58, 60, 69, 71, 73, 74, 76, 78, 80, 91, 92, 94, 96, 98, 101, 114, 119, 121, 137, 139 \\
    \hline
    \multirow{9}{*}{\Centerstack{7.7\\7.8}} & 0, 1, 2, 4, 5, 7, 8, 10, 11, 13, 14, 16, 17, 19, 20, 22, 23, 25, 26, 28, 29, 31, 32, 34, 35, 37, 38, 40, 43, 44, 46, 47, 49, 50, 52, 53, 55, 56, 58, 59, 61, 62, 64, 65, 67, 70, 71, 73, 74, 76, 77, 79, 80, 85, 86, 88, 89, 91, 92, 94, 97, 98, 100, 101, 103, 104, 106, 107, 110, 112, 113, 115, 116, 118, 119, 121, 137, 139, 142, 143, 145, 146, 148, 151, 152, 154, 155, 157, 160, 161, 169, 170, 172, 173, 175,  178, 179, 181, 182, 184, 187, 188, 193, 196, 197, 199, 200, 202, 211, 214, 215, 220, 223, 224, 226, 227,  229, 233, 235, 236, 238, 241, 242, 274, 277, 278, 281, 283, 295, 296, 301, 304, 305, 308, 310, 317, 319,  322, 323, 337, 344, 346, 349, 350, 358, 359, 362, 364, 547, 553, 562, 565, 589, 592, 607, 688, 715 \\
    \hline
    \multirow{28}{*}{8.9} & 0, 1, 2, 3, 4, 5, 6, 7, 8, 9, 10, 11, 12, 13, 14, 15, 16, 17, 18, 19, 20, 21, 22, 23, 24, 25, 26, 27, 28, 29, 30, 31, 32, 33, 35, 36, 37, 38, 39, 40, 41, 42, 43, 44, 45, 46, 47, 48, 49, 50, 51, 52, 53, 54, 55, 56, 57, 58, 59, 60, 62, 63, 64, 65, 66, 67, 68, 69, 70, 72, 73, 75, 76, 78, 79, 81, 82, 83, 84, 85, 86, 87, 88, 89, 90, 91, 92, 93, 94, 95, 96, 97, 99, 100, 101, 102, 103, 104, 105, 107, 109, 110, 111, 112, 113, 114, 115, 117, 118, 119, 120, 121, 122, 123, 125, 126, 127, 129, 130, 131, 132, 133, 134, 135, 136, 137, 139, 142, 143, 146, 147, 148, 149, 150, 151, 153, 154, 155, 157, 158, 159, 161, 162, 163, 164, 165, 166, 167, 171, 172, 173, 174, 175, 177, 178, 179, 181, 182, 183, 185, 186, 188, 189, 190, 191, 192, 195, 196, 197, 199, 200, 201, 202, 203, 204, 205, 207, 208, 209, 210, 211, 212, 213, 214, 215, 216, 217, 219, 220, 222, 224, 225, 226, 227, 228, 231, 232, 233, 234, 235, 236, 237, 238, 239, 241, 242, 243, 244, 246, 247, 248, 249, 253, 254, 256, 257, 258, 260, 261, 263, 264, 266, 267, 269, 270, 271, 274, 275, 278, 281, 282, 284, 287, 289, 291, 292, 293, 294, 299, 304, 305, 306, 307, 308, 309, 311, 312, 314, 315, 317, 320, 321, 323, 324, 326, 329, 333, 334, 336, 340, 342, 346, 347, 349, 351, 359, 360, 361, 363, 365, 369, 372, 373, 374, 377, 378, 380, 381, 382, 383, 386, 390, 394, 396, 397, 400, 401, 402, 405, 406, 412, 414, 415, 419, 420, 427, 429, 431, 434, 435, 439, 442, 445, 447, 449, 455, 456, 457, 461, 462, 469, 470, 472, 473, 475, 477, 493, 495, 496, 497, 498, 500, 504, 506, 507, 513, 515, 519, 520, 521, 523, 525, 529, 533, 539, 540, 544, 547, 552, 557, 558, 560, 564, 565, 567, 568, 576, 583, 584, 591, 593, 598, 601, 611, 613, 620, 621, 623, 637, 641, 643, 646, 649, 652, 653, 656, 658, 659, 667, 668, 670, 679, 681, 685, 694, 697, 699, 703, 704, 716, 721, 724, 736, 775, 784, 787, 791, 796, 807, 810, 816, 822, 830, 837, 845, 855, 865, 867, 880, 933, 935, 939, 954, 975, 986, 994, 1080, 1107, 1134, 1194 \\
    \hline
    \multirow{28}{*}{8.10} & 0, 1, 2, 3, 4, 5, 6, 7, 8, 9, 10, 11, 12, 13, 14, 15, 16, 17, 18, 19, 20, 21, 22, 23, 24, 25, 26, 27, 28, 29, 30, 31, 32, 33, 34, 35, 36, 37, 38, 39, 40, 41, 42, 43, 44, 46, 47, 48, 49, 50, 51, 53, 54, 56, 57, 58, 59, 60, 61, 62, 63, 65, 66, 67, 68, 69, 70, 71, 72, 73, 74, 75, 76, 77, 78, 79, 80, 81, 82, 84, 85, 86, 87, 88, 89, 90, 92, 94, 95, 96, 97, 98, 99, 100, 101, 102, 103, 104, 105, 106, 107, 108, 111, 112, 113, 114, 115, 116, 117, 118, 120, 121, 122, 124, 125, 126, 128, 129, 132, 133, 134, 136, 138, 139, 140, 141, 143, 144, 145, 146, 147, 149, 150, 151, 152, 153, 154, 155, 156, 157, 159, 161, 162, 163, 164, 166, 167, 168, 169, 170, 171, 172, 173, 174, 176, 177, 178, 179, 180, 181, 182, 183, 184, 185, 186, 188, 189, 190, 195, 196, 197, 198, 199, 200, 204, 206, 207, 211, 212, 213, 214, 215, 216, 217, 218, 220, 221, 222, 223, 225, 226, 227, 228, 229, 232, 233, 235, 236, 237, 238, 240, 241, 242, 243, 244, 246, 249, 251, 252, 253, 254, 256, 258, 259, 260, 267, 269, 272, 275, 276, 277, 278, 281, 282, 284, 286, 288, 289, 292, 293, 294, 295, 296, 298, 300, 302, 305, 307, 311, 313, 314, 316, 318, 323, 326, 327, 328, 330, 331, 337, 338, 340, 342, 343, 346, 347, 348, 351, 353, 360, 361, 366, 368, 371, 372, 374, 376, 377, 378, 379, 380, 382, 384, 386, 395, 397, 401, 402, 404, 405, 407, 414, 415, 416, 419, 420, 422, 426, 430, 431, 436, 440, 441, 444, 446, 456, 458, 468, 469, 475, 478, 480, 491, 494, 508, 510, 511, 512, 517, 522, 524, 530, 532, 535, 536, 542, 546, 548, 550, 561, 565, 566, 567, 569, 571, 572, 573, 578, 579, 581, 582, 583, 586, 589, 590, 595, 596, 599, 600, 606, 607, 608, 610, 633, 639, 642, 643, 646, 647, 661, 669, 682, 686, 687, 688, 696, 699, 706, 710, 712, 717, 728, 736, 737, 740, 741, 747, 750, 752, 753, 756, 757, 763, 770, 787, 811, 812, 825, 830, 841, 846, 850, 851, 856, 866, 890, 901, 904, 920, 921, 927, 934, 942, 975, 976, 981, 989, 996, 1005, 1020, 1043, 1049, 1052, 1106, 1109, 1116, 1160, 1260, 1337, 1385\\
    \hline
    \multirow{28}{*}{\Centerstack{9.8\\9.9}} & \numberlist{0, 1, 2, 4, 5, 7, 8, 10, 11, 13, 14, 16, 17, 19, 20, 22, 23, 25, 26, 28, 29, 31, 32, 34, 35, 37, 38, 40, 41, 43, 44, 46, 47, 49, 50, 52, 53, 55, 56, 58, 59, 61, 62, 64, 65, 67, 68, 70, 71, 73, 74, 76, 77, 79, 80, 82, 83, 85, 86, 88, 89, 91, 92, 94, 95, 97, 98, 100, 101, 103, 104, 106, 107, 109, 110, 112, 113, 115, 116, 118, 119, 121, 124, 125, 127, 128, 130, 131, 133, 134, 136, 137, 139, 140, 142, 143, 145, 146, 148, 149, 151, 152, 154, 155, 157, 158, 160, 161, 163, 164, 166, 167, 169, 170, 172, 173, 175, 176, 178, 179, 181, 182, 184, 185, 187, 188, 190, 191, 193, 194, 196, 197, 199, 200, 202, 205, 206, 208, 209, 211, 212, 214, 215, 217, 218, 220, 221, 223, 224, 226, 227, 229, 230, 232, 233, 235, 236, 238, 239, 241, 242, 247, 248, 250, 251, 253, 254, 256, 257, 259, 260, 262, 263, 265, 266, 268, 269, 271, 272, 274, 275, 277, 278, 280, 281, 283, 286, 287, 289, 290, 292, 293, 295, 296, 298, 299, 301, 302, 304, 305, 307, 308, 310, 311, 313, 314, 316, 317, 319, 320, 322, 323, 326, 328, 329, 331, 332, 334, 335, 337, 338, 340, 341, 343, 344, 346, 347, 349, 350, 352, 353, 355, 356, 358, 359, 361, 362, 364, 379, 380, 382, 383, 385, 386, 388, 389, 391, 394, 395, 397, 398, 400, 401, 403, 404, 407, 409, 410, 412, 413, 415, 416, 418, 421, 422, 424, 425, 427, 428, 430, 431, 433, 434, 436, 437, 439, 440, 442, 443, 445, 448, 449, 451, 452, 454, 455, 457, 458, 460, 461, 463, 464, 466, 467, 469, 470, 472, 475, 476, 478, 479, 481, 482, 484, 485, 493, 494, 496, 497, 499, 502, 503, 505, 506, 508, 509, 511, 512, 514, 515, 517, 518, 520, 521, 523, 524, 526, 529, 530, 532, 533, 535, 536, 538, 539, 541, 542, 544, 545, 547, 548, 550, 551, 553, 556, 557, 559, 560, 562, 563, 565, 566, 571, 572, 574, 575, 577, 578, 580, 583, 584, 586, 587, 589, 590, 592, 593, 595, 596, 598, 599, 601, 602, 604, 605, 607, 616, 617, 619, 620, 622, 623, 625, 626, 628, 629, 631, 632, 634, 637, 638, 640, 641, 643, 644, 646, 647, 652, 653, 655, 656, 658, 659, 661, 664, 665, 667, 668, 670, 671, 673, 674, 676, 677, 679, 680, 682, 683, 685, 686, 688, 692, 694, 695, 697, 698, 700, 701, 703, 704, 706, 707, 709, 710, 712, 713, 715, 718, 719, 721, 722, 724, 725, 727, 728, 757, 758, 760, 761, 763, 764, 766, 767, 769, 773, 775, 776, 778, 779, 781, 782, 785, 787, 788, 790, 791, 793, 794, 796, 799, 800, 802, 803, 805, 806, 808, 809, 814, 815, 817, 818, 820, 821, 823, 826, 827, 829, 830, 832, 833, 835, 836, 839, 841, 842, 844, 845, 847, 848, 850, 862, 863, 866, 868, 869, 871, 872, 874, 875, 877, 880, 881, 883, 884, 886, 887, 889, 890, 895, 896, 898, 899, 901, 902, 904, 907, 908, 910, 911, 913, 914, 916, 917, 920, 922, 923, 925, 926, 928, 929, 931, 937, 938, 940, 941, 943, 944, 947, 949, 950, 952, 953, 955, 956, 958, 961, 962, 964, 965, 967, 968, 970, 971, 985, 988, 989, 991, 992, 994, 995, 997, 998, 1001, 1003, 1004, 1006, 1007, 1009, 1010, 1012, 1018, 1019, 1021, 1022, 1024, 1025, 1028, 1030, 1031, 1033, 1034, 1036, 1037, 1039, 1042, 1043, 1045, 1046, 1048, 1049, 1051, 1052, 1058, 1060, 1061, 1063, 1064, 1066, 1069, 1070, 1072, 1073, 1075, 1076, 1078, 1079, 1082, 1084, 1085, 1087, 1088, 1090, 1091, 1093, 1231, 1232, 1234, 1235, 1237, 1240, 1241, 1244, 1246, 1249, 1250, 1252, 1253, 1255, 1261, 1262, 1264, 1267, 1268, 1271, 1273, 1276, 1277, 1279, 1280, 1282, 1285, 1286, 1288, 1289, 1291, 1294, 1295, 1303, 1304, 1306, 1307, 1309, 1312, 1313, 1315, 1316, 1318, 1321, 1322, 1325, 1327, 1330, 1331, 1333, 1334, 1336, 1354, 1357, 1358, 1360, 1361, 1363, 1366, 1367, 1369, 1370, 1372, 1375, 1376, 1384, 1385, 1387, 1388, 1390, 1393, 1394, 1396, 1397, 1399, 1402, 1403, 1406, 1408, 1411, 1412, 1414, 1415, 1417, 1424, 1426, 1429, 1430, 1433, 1435, 1438, 1439, 1441, 1442, 1444, 1447, 1448, 1450, 1451, 1453, 1456, 1457, 1514, 1516, 1519, 1520, 1522, 1523, 1525, 1528, 1529, 1531, 1532, 1534, 1537, 1538, 1546, 1547, 1549, 1550, 1552, 1555, 1556, 1558, 1559, 1561, 1564, 1565, 1570, 1573, 1574, 1576, 1577, 1579, 1600, 1601, 1603, 1604, 1606, 1609, 1610, 1612, 1613, 1615, 1618, 1619, 1627, 1628, 1630, 1631, 1633, 1636, 1637, 1639, 1640, 1642, 1645, 1646, 1651, 1654, 1655, 1657, 1658, 1660, 1669, 1672, 1673, 1678, 1681, 1682, 1684, 1685, 1687, 1690, 1691, 1693, 1694, 1696, 1699, 1700, 1723, 1726, 1727, 1732, 1735, 1736, 1738, 1739, 1741, 1750, 1753, 1754, 1759, 1762, 1763, 1765, 1766, 1768, 1771, 1772, 1774, 1775, 1777, 1780, 1781, 1790, 1792, 1793, 1795, 1798, 1799, 1801, 1802, 1804, 1807, 1808, 1813, 1816, 1817, 1819, 1820, 1822, 1894, 1897, 1898, 1900, 1901, 1903, 1912, 1915, 1916, 1921, 1924, 1925, 1927, 1928, 1930, 1934, 1936, 1937, 1939, 1942, 1943, 1969, 1970, 1975, 1978, 1979, 1981, 1982, 1984, 1993, 1996, 1997, 2002, 2005, 2006, 2008, 2009, 2011, 2015, 2017, 2018, 2020, 2023, 2024, 2035, 2036, 2038, 2042, 2044, 2045, 2047, 2050, 2051, 2056, 2059, 2060, 2062, 2063, 2065, 2092, 2096, 2098, 2099, 2101, 2104, 2105, 2116, 2117, 2119, 2123, 2125, 2126, 2128, 2131, 2132, 2137, 2140, 2141, 2143, 2144, 2146, 2158, 2159, 2164, 2167, 2168, 2170, 2171, 2173, 2177, 2179, 2180, 2182, 2185, 2186, 2461, 2464, 2465, 2468, 2470, 2482, 2483, 2488, 2491, 2492, 2495, 2497, 2501, 2504, 2506, 2509, 2510, 2522, 2524, 2528, 2531, 2533, 2536, 2537, 2542, 2545, 2546, 2549, 2551, 2650, 2653, 2654, 2657, 2659, 2666, 2668, 2671, 2672, 2707, 2708, 2711, 2713, 2725, 2726, 2731, 2734, 2735, 2738, 2740, 2747, 2749, 2752, 2753, 2767, 2774, 2776, 2779, 2780, 2785, 2788, 2789, 2792, 2794, 2830, 2833, 2834, 2848, 2855, 2857, 2860, 2861, 2866, 2869, 2870, 2873, 2875, 2888, 2893, 2896, 2897, 2900, 2902, 2909, 2911, 2914, 2915, 3028, 3031, 3032, 3035, 3037, 3076, 3077, 3091, 3098, 3100, 3103, 3104, 3112, 3113, 3116, 3118, 3139, 3140, 3143, 3145, 3152, 3154, 3157, 3158, 3199, 3220, 3221, 3224, 3226, 3233, 3235, 3238, 3239, 3260, 3262, 3265, 3266, 3274, 3275, 3278, 3280, 4921, 4927, 4936, 4939, 4963, 4966, 4975, 4981, 5044, 5047, 5056, 5062, 5089, 5098, 5101, 5299, 5305, 5332, 5341, 5344, 5413, 5422, 5425, 5452, 5467, 6151, 6196, 6439, 6520\\} \\
    \hline
\end{xltabular}
\renewcommand{\arraystretch}{1}
\setlength{\tabcolsep}{6pt}

\newpage
\subsection{Expansion Searches}

\renewcommand{\arraystretch}{1.2}
\setlength{\tabcolsep}{6pt}
\begin{xltabular}{\linewidth}{|c|c|c|c|c|c|c|c|c|c|c|c|}
    \caption{Expansion Searches for monomials $M(x)$, with coefficients over $\F_{3^m}$.}
    \label{tab:searches_times}
    \\
    \hline
    \multirow{3}{*}{n} & \multirow{3}{*}{$M$} & \multicolumn{10}{c|}{Expansion Terms} \\ \cline{3-12}
    & & \multicolumn{2}{c|}{1} & \multicolumn{2}{c|}{2} & \multicolumn{2}{c|}{3} & \multicolumn{2}{c|}{4} & \multicolumn{2}{c|}{5} \\
    \cline{3-12}
    & & m & time & m & time & m & time & m & time & m & time \\
    \hline
    \endhead

    \multirow{2}{*}{6} & $x^2$ & 6 & 0.1s & 6 & 41m & 6 & 535.5h & 3 & 33.7h & 2 & 411.1h \\
    & $x^{10}$ & 6 & 1.8s & 6 & 63.1h & 3 & 33.1h & 2 & 29.32h & 1 & 2h \\
    \hline
    \multirow{4}{*}{7} & $x^2$ & 7 & 10ms
    & 7 & 8m & - & - & - & - & - & - \\
    & $x^4$ & 7 & 5.6s & 7 & 34.4h & - & - & - & - & - & - \\
    & $x^{10}$ & 7 & 2.8m & 7 & 455.7h & - & - & - & - & - & - \\
    & $x^{28}$ & 7 & 4.2m & 7 & 589.9 h & - & - & - & - & - & - \\
    \hline
    \multirow{5}{*}{8} & $x^2$ & 8 & 19s & 4 & 21.3m & 4 & 48.4h & 2 & 15.24h & 1 & 3.2h \\
    & $x^4$ & 8 & 90ms & 8 & 29.6h & 4 & 60.7h & 2 & 45.7h & 1 & 10.4h \\
    & $x^{10}$ & 8 & 0.1s & 8 & 5.4h & 4 & 52.9h & 2 & 48.5h & 1 & 4.5h \\
    & $x^{28}$ & 8 & 0.1s & 8 & 352.4h & 4 & 323.5h & 2 & 55.6h & 1 & 4.7h \\
    & $x^{82}$ & 8 & 36s & 8 & 42.4h & 2 & 148.4h & 1 & 131.7h & 1 & 304.8h\\
    \hline
    \multirow{6}{*}{9} & $x^2$ & 1 & 120ms & 1 & 13.6h & 1 & 42.9h & 1 & 185.2h & - & - \\
    & $x^4$     & 1 & 38s & 1 & 24.3m & 1 & 277.5h & 1 & 85.5h  & - & - \\
    & $x^{10}$  & 1 & 92s & 1 & 1.1h  & 1 & $>$ 300h & 1 & 205.1h  & - & - \\
    & $x^{28}$  & 1 & 8s  & 1 & 9.1m  & 1 & 87.91h & 1 & 168.6h & - & - \\
    & $x^{82}$  & 1 & 10m & 1 & 1.7h  & 1 & $>$ 300h & 1 & 93.9h   & - & - \\
    & $x^{244}$ & 1 & 11m & 1 & 1.9h  & 1 & 294.3h & 1 & 89.9h  & - & - \\
    \hline
\end{xltabular}
\renewcommand{\arraystretch}{1}
\setlength{\tabcolsep}{6pt}

\end{document}